\documentclass[11pt,reqno]{amsart}
\usepackage{hyperref}

\allowdisplaybreaks[4]
\usepackage{lipsum}
\usepackage{amsfonts,cleveref}
\usepackage{graphicx}
\usepackage{epstopdf}
\usepackage{algorithmic}
\usepackage{amsmath,amssymb}
\usepackage{dsfont}
\usepackage{enumitem}
\setlist[enumerate]{leftmargin=.5in}
\setlist[itemize]{leftmargin=.5in}

\usepackage{geometry}
\usepackage{color}
\theoremstyle{plain}

\newtheorem{lemma}{Lemma}[section]
\newtheorem{theorem}[lemma]{Theorem}

\newtheorem{proposition}[lemma]{Proposition}
\newtheorem{assumption}{Assumption}

\crefname{assumption}{Assumption}{Assumptions}
\crefname{hypothesis}{Hypothesis}{Hypotheses}
\crefname{proposition}{Proposition}{Propositions}
\crefname{theorem}{Theorem}{Theorems}
\crefname{lemma}{Lemma}{Lemmas}

\newcommand{\dd}{\mathrm d}
\newcommand{\E}{\mathbb E}
\newcommand{\N}{\mathbb N}
\newcommand{\R}{\mathbb R}
\newcommand{\B}{\mathbb B}

\newtheorem{example}{Example}

\begin{document}
\title[Stochastic wave equation with dissipative damping and full discretization]{Long-time dynamics of stochastic wave equation with dissipative damping and its full discretization: exponential ergodicity and strong law of large numbers}

\author{Meng Cai, Chuchu Chen, Jialin Hong, Tau Zhou}
\address{LSEC, ICMSEC, Academy of Mathematics and Systems Science, Chinese Academy of Sciences, Beijing 100190, China} 
  \email{mcai@lsec.cc.ac.cn}
\address{LSEC, ICMSEC, Academy of Mathematics and Systems Science, Chinese Academy of Sciences, Beijing 100190, China;
 School of Mathematical Sciences, University of Chinese
Academy of Sciences, Beijing 100049, China} 
  \email{chenchuchu@lsec.cc.ac.cn, hjl@lsec.cc.ac.cn}
\address{Department of Applied Mathematics, The Hong Kong Polytechnic University, Hung Hom, Kowloon, Hong Kong} 
  \email{tau.zhou@polyu.edu.hk}

\thanks{This work is funded by the National key R$\&$D Program of China under Grant (No. 2020YFA0713701), National Natural Science Foundation of China (No. 12031020, No. 11971470, No. 12022118, No. 11871068), and by Youth Innovation Promotion Association CAS, China.}
\begin{abstract}
For stochastic wave equation, when the dissipative damping is a non-globally Lipschitz function of the velocity, there are few results on the long-time dynamics, in particular, the exponential ergodicity and strong law of large numbers, for the equation and its numerical discretization to our knowledge. Focus on this issue, the main contributions of this paper are as follows. First, based on constructing novel Lyapunov functionals, we show the unique invariant measure and exponential ergodicity of the underlying equation and its full discretization. Second, the error estimates of invariant measures both in Wasserstein distance and in the weak sense are obtained. Third, the strong laws of large numbers of the equation and the full discretization are obtained, which states that the time averages of the exact and numerical solutions are shown to converge to the ergodic limit almost surely. \end{abstract}

\maketitle

\section{Introduction}
In this paper, we consider  the following stochastic wave equation with dissipative damping
\begin{align}\label{SWE}
\begin{cases}
\mathrm{d} u(t)  = v(t)    \mathrm{d} t,\\
\mathrm{d} v(t)  = - \Lambda u(t) \mathrm{d} t -\eta v(t) \mathrm{d} t
    - F( v(t) ) \mathrm{d} t + \mathrm{~d} W(t),\\
u(t)|_{\partial \mathcal O} = 0, u(0) = u_0, v(0) = v_0,
\end{cases}
\end{align}
where  $\mathcal O \subset \mathbb R^d \;(d \le 3)$ is a bounded open domain with regular boundary $\partial \mathcal O$, and $- \Lambda := \Delta$ is the Dirichlet Laplacian in
$H : = L^2(\mathcal{O} ; \mathbb{R})$.
Here $\eta v$ is a linear damping with $\eta$ being a positive constant, and $F$ is  the dissipative damping, which is a non-globally
Lipschitz function of the velocity.
The stochastic process $W(t),\;t\geq 0$ is an $H$-valued $Q$-Wiener process on a filtered probability space
$(\Omega, \mathcal{F}, \{ \mathcal{F}_t \}_{t \geq 0}, \mathbb{P})$ with $Q \in \mathcal{L}(H)$ being symmetric, positive definite and of trace class. 
The equation \eqref{SWE}, which is first proposed in \cite{pardoux1975equations}, characterizes the displacement field of a particle suspended in a continuous media while being forced by random perturbations via an additive Gaussian noise, for instance, the motion of a vibrating string or the motion of a strand of
DNA in a fluid.

In recent years, the study on the long-time dynamics, including the invariant measure, the ergodicity as well as the strong law of large numbers for stochastic partial differential equations and its numerical discretization has drawn a lot of attention (see e.g. \cite{Brehier2014approximation,Brehier2022approximation,Brehier2017approximation,Brehier2016high,Caraballo2024convergence,Cerrai2006on,Chen2023CLT,Chen2017approximation,Chen2020full,Cui2021weak,Hong2017high,Hong2019invariant,Hong2017numerical,Lei2023numerical,Martirosyan2014exponential,Zhou2005random}). 
Especially, for the stochastic damped Klein–Gordon
equation, where the term $F$ is a function of the displacement $u$ (it is called a reaction term),
\cite{Cerrai2006on,Martirosyan2014exponential} show that the equation possesses a unique invariant measure and is exponentially ergodic; the authors of \cite{Lei2023numerical} propose a full discretization by a spectral Galerkin method in space and an exponential Euler integrator in time to inherit the ergodicity of  the stochastic damped Klein–Gordon
equation, and obtain the convergence rate of invariant measure  via weak error estimates of the full discretization.
To our knowledge, there are few results on the long-time dynamics, in particular, the exponential ergodicity and strong law of large numbers, of the
equation \eqref{SWE} and its numerical discretization.  This paper aims to take a step further and fill this gap. We consider the following questions:
\begin{itemize}
\item[(I)] Does the equation \eqref{SWE} (resp. its full discretization)  admit  a unique invariant measure (resp. a unique numerical invariant measure), and further exponential ergodicity?
\item[(II)] If so, does the numerical invariant measure converge to the original one? And in which sense?
\item[(III)] Further, do the time averages of the exact and numerical solutions converge to the ergodic limit in the almost sure sense? Namely, do strong laws of large numbers hold?
\end{itemize}

For the question (I),
the existence of an invariant measure of the equation \eqref{SWE}  is proved in \cite{Barbu2007stochastic,Kim2008on} by employing the classical Krylov–Bogoliubov argument.
The uniqueness of the invariant measure of the equation \eqref{SWE}
is given in \cite{Barbu2007stochastic} by showing that regardless of initial conditions, the solutions always
converge to one another as time tends to infinity.
 The polynomial ergodicity of the equation \eqref{SWE} is obtained in \cite{MR4680506}, relying on a combination
of Lyapunov conditions, the contracting property of the Markov transition semigroup, and  $d$-small sets.
For the numerical study of \eqref{SWE}, we are only aware of the paper \cite{Cai2023strong}, where  the strong convergence  of a  full discretization on finite time is analyzed. 

In this paper, we consider the full discretization of \eqref{SWE} by applying the spectral Galerkin method in space
and the backward Euler method in time, i.e.,
\begin{align}\label{fulldis-intro}
\begin{cases}
u^N_{n+1} = u^N_n + \tau \, v^N_{n+1}, \\
v^N_{n+1} = v^N_n - \tau (\Lambda_N u^N_{n+1} + \eta v^N_{n+1} + \Pi_N F( v^N_{n+1}) ) + \Pi_N \delta W_n.
\end{cases}
\end{align} 
First, we  construct  novel Lyapunov functionals to obtain the uniform moment boundedness and  exponential contraction properties of the exact and numerical solutions, i.e., there exists a constant $C$ independent of time such that
\begin{align*}
\sup_{t\geq 0}{\mathbb E}[\|X(t)\|^{p}_{\mathbb{H}^{2}}]+\sup_{m\in{\mathbb N}}\frac{1}{m}\sum_{k=0}^{m}{\mathbb E}[\|X_k^N\|^{2}_{\dot{\mathbb H}^{2}}]  &\leq C,\quad p\geq 1\\
\|X(t_n)-\widetilde{X}(t_n)\|_{L^p(\Omega;{\mathbb H}^1)}
+\|X_n^N-\widetilde{X}_n^N\|_{L^p(\Omega;{\mathbb H}^1)}
&\leq \frac{1+2\epsilon}{1-2\epsilon}e^{-\epsilon t_n}\|X(0)-\widetilde{X}(0)\|_{L^p(\Omega;{\mathbb H}^1)}
\end{align*}
where $X(t)=(u(t),v(t))^{\top}$, $X_k^N=(u_k^N,v_k^N)^{\top}$ and $\epsilon$ is a sufficiently small positive constant. Here $X(t)$ (resp. $X_n^N$) and $\widetilde{X}(t)$ (resp. $\widetilde{X}_n^N$) are the solutions of \eqref{SWE} (resp. \eqref{fulldis-intro}) with different initial data $X_0$ and $\widetilde{X}_0$, respectively.
 Based on these results, we finally show the exponential ergodicity of the underlying equation and its full discretization, which gives a positive answer to the question (I). This means that the full discretization inherits the exponential ergodicity of the original equation.

To study the question (II), 
we first estimate the strong error of the full discretization \eqref{fulldis-intro} and obtain that
$$
\|X(t_k)-X_k^N\|_{L^p(\Omega; {\mathbb H}^1)}\lesssim \lambda_N^{-\frac12}(1+t_k^{\frac12})+\tau^{\frac12}(1+t_k^{3}),
$$
 where $\lambda_N$ is the $N$th eigenvalue of the operator $\Lambda$ and $\tau$ is the time step-size. We note that in the above result on the estimate of the error, the growth with respect to time is at most polynomial. Making utilize of the exponential ergodicity of the numerical solution, we then obtain the error estimates of invariant measures both in Wasserstein distance and in
the weak sense, which provides the answer to the question (II). More precisely, for the error between invariant measures in Wasserstein distance, we obtain the following convergence result
$$
{\mathcal W}_p(\mu,\mu_{\tau}^N)\lesssim \lambda_N^{-\frac12}+\tau^{\frac12}.
$$
And for the weak error of invariant measures, we obtain the following estimate
\begin{align*}
\Big| \int_{\mathbb H^1} \varphi \dd \mu - \int_{\mathbb H^1_N} \varphi \dd \mu_\tau^N \Big|\lesssim  e^{- \epsilon \gamma t_n}+  \lambda_N^{-\frac \gamma2}(1+t_n^{\frac\gamma2})+  \tau^{\frac{\gamma+p/2}2}(1+t_n^{3(\gamma+\frac p2)})
\end{align*}
for any $n\in{\mathbb N}$ and for certain test functional $\varphi\in C_{p,\gamma}$, which yields the weak convergence of the numerical invariant measure, i.e.,
$
\lim_{\tau\to 0}\lim_{N\to\infty} \int_{\mathbb H^1_N} \varphi \dd \mu_\tau^N= \int_{\mathbb H^1} \varphi \dd \mu.
$
This reveals that the numerical invariant measure of the full discretization can approximate the one of the original equation properly.

To solve the question (III),  based on the Markov property of the solution and the property of conditional expectation, we  estimate the error between
the time averages of the exact and numerical solutions and the ergodic limit $\mu(\varphi):=\int_{{\mathbb H}^1}\varphi\dd \mu$, and obtain the following results
\begin{align*}
&\left|\frac{1}{t}\int_0 ^{t}\varphi(X(t))\dd t-\mu(\varphi)\right| \lesssim t^{-\frac12+\varepsilon}\quad \text{a.s.},
 \\
&\left|\frac1n\sum_{k=1}^n\varphi(X^N_k)-\mu(\varphi)\right|
\lesssim \tau^{\frac{\gamma+p/2}{2}-\varepsilon}\big(1+t_n^{3(\gamma+\frac p2)}\big)^{1+\varepsilon}+\lambda_N^{-\frac\gamma2+\varepsilon}\big(1+t_n^{\frac\gamma2}\big)^{1+\varepsilon}+t_n^{-\frac12+\varepsilon}\quad \text { a.s.}
\end{align*}
for $\varepsilon>0$. 
These then lead to  the a.s. convergence of the time averages of the exact and numerical solution to  the ergodic limit, namely, the
strong laws of large numbers of the equation \eqref{SWE} and the full discretization \eqref{fulldis-intro} hold. This answers the question (III). The result of the strong law of large numbers for the full discretization illustrates the effectiveness of constructing the time-average of the numerical solution to approximate the ergodic limit using a single sample path of the numerical solution, which can greatly improve computational efficiency by avoiding simulating a large number of samples.

The outline of this paper is organized as follows.
The next section presents some preliminaries for investigating \eqref{SWE} and its full discretization.
 \Cref{Sec:well-posedness} is devoted to the moment estimates of the exact and numerical solutions.
In \Cref{Sec:ergodicity}, we obtain the exponential ergodicity of both the underlying equation and full discretization.
The estimates on the approximation of the invariant measure are presented  in
\Cref{Sec:approximation}.

\section{Preliminaries}

In this section, we present some preliminaries for investigating the stochastic wave equation \eqref{SWE} and its full discretization.

\subsection{Notations}

Throughout this paper, we will use the symbol $C$ to denote any unspecified positive constant independent of mesh size, whose value may be updated throughout the proofs. When we want to specify the dependence of $C$ on some specific quantities, we will indicate them through a subscript or use $C(\cdot, \cdot )$.

Denote by $L^p(\mathcal{O};\mathbb{R}), p \ge 1$ ($L^p$ for short) the usual Lebesgue space consisting of $p$th integrable functions.
When $p=2$, the space $H$ is then a Hilbert space with  inner product $\langle \cdot, \cdot\rangle$ and norm $\|\cdot\|$.

Note that there exists a family of eigenpair
$ \{ \lambda_k, e_k \}_{k \in \mathbb N}$ of $\Lambda$ in $H$ such that
$ \Lambda e_k=\lambda_k e_k$
for an increasing sequence of positive numbers
$\{\lambda_k\}_{k \in \mathbb N}$ tending to infinity.
More precisely, there exists $C_d \in(0, \infty)$ such that $\lambda_n \sim C_d n^{2 / d}$ when $n \rightarrow \infty$.
Setting
$\dot{H}^r = \operatorname{Dom}(\Lambda^{\frac r2}), r \in \mathbb{R}$,
endowed with the inner product
$$ \langle x, y \rangle_{\dot{H}^r}
      =  \langle  \Lambda^{\frac r2} x, \Lambda^{\frac r2} y \rangle
      =  \sum_{k = 1}^\infty \lambda_k^r
        \langle x, e_k \rangle \langle y, e_k \rangle$$
and the induced norm
$$ \| x \|_{\dot{H}^r}  = \left( \sum_{k = 1}^{\infty} \lambda_k^r
   \big|  \langle x, e_k \rangle \big|^2 \right)^{\frac 12},$$
it can be easily shown that the embedding $\dot{H}^{r_1} \hookrightarrow \dot{H}^{r_2}$ is compact for $r_1 > r_2$.
We denote by $\mathbb{H}^\beta = \dot{H}^\beta \times \dot{H}^{\beta-1}, \beta \in \mathbb{R}$ the product space endowed with the inner product
$$  \langle X, Y \rangle_{\mathbb{H}^\beta}
       = \langle x_1,y_1 \rangle_{\dot{H}^\beta}
        + \langle x_2,y_2 \rangle_{\dot{H}^{\beta-1}}, \,
        X =(x_1,x_2)^\top\in \mathbb{H}^\beta,\;
        Y =(y_1,y_2)^\top\in \mathbb{H}^\beta$$
and the norm
$$ \| X \|_{\mathbb{H}^\beta}
     = \Big( \| x_1 \|_{\dot{H}^\beta}^2
              + \| x_2 \|_{\dot{H}^{\beta-1}}^2 \Big)^{\frac12},\,
     X =(x_1,x_2)^\top \in \mathbb{H}^\beta.$$

In the sequel, concerning the well-posedness of \eqref{SWE}, we shall give some necessary assumptions.

\subsection{Stochastic wave equation setting}

Setting $ X(t) = (u(t),v(t))^\top $,
we may rewrite \eqref{SWE} in a compact form
\begin{equation}\label{SWE1}
\dd X(t) = A X(t) \, \dd t - {\bf F}(X(t)) \, \dd t + B \dd W(t),
 \quad
 X(0) = X_0,
\end{equation}
where
\begin{equation*}
A= \left[  \begin{array}{ccc}   0 & I \\  -\Lambda & 0
                       \end{array} \right]
,
\:
{\bf F}(X(t))= \left[ \begin{array}{ccc} 0 \\ \eta v(t) + F(v(t)) \end{array} \right]
,
\:
B = \left[ \begin{array}{ccc}  0 \\  I   \end{array} \right],
\:
X_0 =
 \left[
 \begin{array}{ccc}
     u_0 \\ v_0
 \end{array}
 \right].
\end{equation*}
The operator $A$ with domain
\begin{equation*}
\operatorname{Dom}(A)= \{ X \in \mathbb{H}:
  A X = ( v, - \Lambda u )^\top \in \mathbb{H} \}  =  \mathbb{H}^1
\end{equation*}
is the generator of  a unitary group
$E(t)$ on $\mathbb{H}^1$, given by
\begin{equation*}
E(t)= e^{t A} =
\left[
 \begin{array}{ccc}
     \cos( t \Lambda^{\frac12}) &
     \Lambda^{-\frac12} \sin( t \Lambda^{\frac12}) \\
     -\Lambda^{\frac12} \sin( t \Lambda^{\frac12}) &
     \cos( t \Lambda^{\frac12})
 \end{array}
  \right], \, t \in \mathbb{R}.
\end{equation*}

Next, we state the main assumptions on the nonlinearity and the noise.

\begin{assumption}\label{ass:nonlinearity}
Let $F: H \to H$ be the Nemytskii operator associated to a continuously differentiable function $f: \mathbb{R} \to \mathbb{R}$, given by
\begin{equation*}
F(v)(x):=f(v(x)), v \in H, x \in \mathcal O,
\end{equation*}
satisfying
\begin{align}
\left| f (\xi) \right| & \leq a_1( 1 + |\xi| ),
\, a_1 \in (0,\tfrac{\sqrt{2}}{2} \eta),\label{eq:linear-growth} \\
\inf _{\xi \in \mathbb{R}} f'(\xi) & = a_2 >-\eta. \label{inf-derivative}
\end{align}
\end{assumption}
It follows from \eqref{inf-derivative} that
\begin{align}\label{oneside-lip}
\langle v_1 - v_2, F(v_1) - F(v_2) \rangle
     \ge a_2 \| v_1 - v_2 \|^2.
\end{align}

Below we provide an example of the non-globally Lipschitz function $f$ which satisfies \cref{ass:nonlinearity}.

\begin{example}\label{ex-nonlip}
Let $h\in C(\R)$ be defined by
\begin{align*}
 h(x)=\begin{cases}
\frac{\alpha_n-x}{2}, \quad &x \in\left[\alpha_n, \alpha_n+n\right), \\
\frac{n\left(x-\alpha_{n+1}\right)}{2}, \quad &x \in\left[\alpha_n+n, \alpha_{n+1}\right), \\
-h(-x),\quad &x\le0,
\end{cases}
\end{align*}
where $\alpha_n=\frac{n(n+1)}{2}-1, ~ n \geq 1.$ Then we have $|h(x)| \le |x|$ for any $x\in\R$ and $ h^{\prime}(x) \ge -\frac12>-1$ for any $x\notin \{ -\alpha_n,-\alpha_n-n,\alpha_n,\alpha_n+n\}_{n\ge 1}$.

For any point $y$ belongs to the set $\{ -\alpha_n,-\alpha_n-n,\alpha_n,\alpha_n+n\}_{n\ge 1}$ of non-differentiable points for $h$, by selecting appropriate small positive constants $\varepsilon_1(y)$ and $\varepsilon_2(y)$, one can modify the function $h$ within that small interval using an arc segment which tangent to $h$ at points $(y-\varepsilon_1(y),h(y-\varepsilon_1(y)))$ and $(y+\varepsilon_2(y),h(y+\varepsilon_2(y)))$. Denoted by $h_m$ the function obtained by modifying $h$ in the above way, then $f:=\frac{\eta}{2} h_m$ satisfies conditions \eqref{eq:linear-growth} and \eqref{inf-derivative}. However, $f$ is not a Lipschitz function.
\end{example}

Recall that the stochastic process $W(t),\;t\geq 0$ is an $H$-valued $Q$-Wiener process on a filtered probability space
$(\Omega, \mathcal{F}, \{ \mathcal{F}_t \}_{t \geq 0}, \mathbb{P})$ with $Q \in \mathcal{L}(H)$ being symmetric, positive definite and of trace class. It has an expansion
\begin{equation*}
 W(t) = \sum_{k = 1}^\infty  Q^{\frac12} e_k \beta_k(t), \, t\ge 0,
\end{equation*}
where 
$\{ e_k \}_{k \geq 1}$ forms an orthonormal basis in $H$
and $\{ \beta_k (t) \}_{k \geq 1}$ is a sequence of independent and identically distributed real-valued Brownian motions.

\begin{assumption}\label{ass:noise}
There exists a sequence of positive numbers $\{ q_k \}_{k \in \mathbb N}$ such that $Q: H \rightarrow H$ is diagonalizable with respect to the orthonormal basis $\{ e_k \}_{k \in \mathbb N}$, i.e., of the form
 \begin{equation*}
 Q e_k = q_k e_k, \, k \in \mathbb N,
 \end{equation*}
and that $\sum_{k\in{\mathbb N}}q_k\|e_k\|^2_{\infty}<\infty
$ and
\begin{equation}\label{eq:A-Q}
\big\| \Lambda^{\frac 12} Q^{\frac 12} \big\|_{\mathcal{L}_2(H)}
  = \left( \sum_{k=1}^\infty \lambda_k q_k \right)^{\frac 12}
  = \Big( \mathrm{Tr} (\Lambda Q) \Big)^{\frac 12}  < \infty.
\end{equation}
\end{assumption}

Here and below,  let $\mathcal{L}_2(H)$ be the set of Hilbert--Schmidt operators with norm
\begin{equation*}
\| T \|_{\mathcal{L}_2(H)}
:=\Big( \sum_{k=1}^\infty  \|  T e_k  \|^2 \Big)^{\frac12}.
\end{equation*}
Furthermore, the set $\mathcal{L}_2^0$ denotes the space of Hilbert--Schmidt operators from $Q^{\frac12}(H)$ to $H$ with norm
$ \| \cdot \|_{\mathcal{L}_2^0} : =  \| \cdot  Q^{\frac12}  \|_{\mathcal{L}_2 (H)}$.

Under the above assumptions,
we can easily obtain the well-posedness of the stochastic wave equation \eqref{SWE1} by following the approach via Yosida approximations in the proof of
\cite[Theorem 2.3]{Barbu2007stochastic}.
The details are thus omitted.
\begin{lemma}
Let \cref{ass:nonlinearity,ass:noise} hold and
$X_0 \in \mathbb{H}^1$.
Then there exists a unique mild solution of \eqref{SWE1} in $\mathbb{H}^1$,
given by,
for all $t \ge 0$,
\begin{equation}
X(t) = E(t) X_0 - \int_0^t E (t-s) {\bf F}(X(s)) \dd s
       + \int_0^t E (t-s) B \dd W(s), \,\, a.s.
\end{equation}
\end{lemma}

\subsection{Full discretization}
In this subsection, we propose a full discretization of \eqref{SWE1} by means of the spectral Galerkin method in space and the backward Euler method in time.

\subsubsection{Spectral Galerkin method}
Taking any positive integer $N \in \N$,
we define a finite-dimensional subspace $H_N$ of $H$ as
$$H_N := \text{span} \{ e_1, e_2, \cdots, e_N \},$$
where the bases $\{ e_j \}_{j=1}^N$ are the first $N$ eigenfunctions of the linear operator $\Lambda$.
We next introduce the projection operators  $ \Pi_N: \dot{H}^{\alpha}\to H_N$ and $\mathbf \Pi_N:\mathbb{H}^{\beta} \to H_N\times H_N$, which are  respectively
 given by
\begin{equation*}
\begin{split}
\Pi_N  \psi & = \sum_{i=1}^N \langle \psi, e_i \rangle e_i, \,
 \psi \in \dot{H}^{\alpha}, \, \alpha \geq 0, \\
\mathbf\Pi_N X & = (\Pi_N x_1,  \Pi_N x_2 )^\top, \,
X=(x_1,x_2)^\top \in \mathbb{H}^{\beta},\, \beta \geq 1.
\end{split}
\end{equation*}

One can immediately verify that
\begin{equation}\label{eq:P_N-estimate}
 \| ( \Pi_N - I)  \psi \| \leq \lambda_N^{-\frac \gamma2}
 \| \psi \|_{\dot{H}^\gamma}, \,
 \psi \in \dot{H}^\gamma, \, \gamma \ge0.
\end{equation}
We define the discrete Laplacian $\Lambda_N: H_N \rightarrow H_N $  by
\begin{equation*}
\Lambda_N \psi =  \Lambda \Pi_N \psi = \Pi_N \Lambda \psi
   = \sum_{i=1}^N \lambda_i \langle \psi, e_i \rangle e_i,
   \, \psi \in H_N.
\end{equation*}
Then the spectral Galerkin method of \eqref{SWE} can be expressed as 
\begin{align}\label{eq:spectral-galerkin}
\begin{cases}
\mathrm{d} u^N(t)  = v^N(t) \mathrm{d} t,\\
\mathrm{d} v^N(t)  = - \Lambda_N u^N(t) \mathrm{d} t
   - \eta v^N(t) \mathrm{d} t   - F_N (v^N(t)) \mathrm{d} t
   + \Pi_N \mathrm{d} W(t),\\
u^N(0)= u^N_0, v^N(0)=v^N_0,
\end{cases}
\end{align}
which can also be written in a compact form
\begin{equation*}
    \dd X^N(t) = A_N X^N(t) \, \dd t
       - {\bf F}_N (X^N(t)) \, \dd t   + B_N \dd W(t),
\end{equation*}
provided $X^N(0)=(u_0^N, v_0^N)^\top$, $F_N = \Pi_N F$ and
\begin{equation*}
X^N  = \left[
 \begin{array}{ccc}
     u^N \\ v^N
 \end{array}
 \right]
,
\:
 A_N = \left[
 \begin{array}{ccc}
     0 & I \\
     -\Lambda_N & 0
 \end{array}
 \right]
,
\:
{\bf F}_N(X^N) = \left[
 \begin{array}{ccc}
     0 \\
    \eta v^N + F_N (v^N)
 \end{array}
 \right]
,
\:
B_N = \left[
 \begin{array}{ccc}
     0 \\ \Pi_N
 \end{array}
 \right].
\end{equation*}
Similarly, the operator $A_N$ generates a unitary group $E_N(t),t\in{\mathbb R}$ on $H_N \times H_N$, which is given by
\begin{equation*}
E_N(t)= e^{t A_N} =
\left[
 \begin{array}{ccc}
     \cos( t \Lambda_N^{\frac12}) &
     \Lambda_N^{-\frac12} \sin( t \Lambda_N^{\frac12}) \\
     -\Lambda_N^{\frac12} \sin( t \Lambda_N^{\frac12}) &
     \cos( t \Lambda_N^{\frac12})
 \end{array}
  \right].
\end{equation*}
The mild solution of \eqref{eq:spectral-galerkin} reads
\begin{equation}\label{eq:mild-space}
X^N(t) = E_N(t) X^N(0) - \int_0^t E_N(t-s) {\bf F}_N(X^N(s)) \, \dd s
                 + \int_0^t E_N(t-s) B_N \, \dd W(s).
\end{equation}

\subsubsection{Full discretization}
The full discretization is obtained by applying the backward Euler method in the temporal direction of \eqref{eq:spectral-galerkin}.
Let $\tau \in (0,1)$ be the uniform time step-size.
For any nonnegative integer $n \in \N$, $t_n := n \tau$ and the increment of the Wiener process is denoted by
$\delta W_n = W(t_{n+1}) - W(t_n)$.
The full discretization of \eqref{SWE} reads as 
\begin{align}
\begin{cases}\label{eq:full-scheme}
u^N_{n+1} = u^N_n + \tau \, v^N_{n+1}, \\
v^N_{n+1} = v^N_n - \tau (\Lambda_N u^N_{n+1} + \eta v^N_{n+1}
+  \Pi_N F( v^N_{n+1}) ) +\Pi_N \delta W_n.
\end{cases}
\end{align}
Let $X^N_n = (u_n^N,v^N_n)^\top$. The compact form of \eqref{eq:full-scheme} is
\begin{equation*}
X_{n+1}^N = X_n^N + \tau A_N X_{n+1}^N
      - \tau \mathbf{F}_N (X_{n+1}^N)
      + B_N \delta W_n.
\end{equation*}

The full discretization \eqref{eq:full-scheme} is well-defined and a.s. uniquely solvable
thanks to the uniform monotonicity theorem in
\cite[Theorem C.2]{Stuart1996dynamical}.
The proof is trivial and thus omitted.
\begin{lemma}
Let \cref{ass:nonlinearity,ass:noise} hold and
$ X_0 \in \mathbb H^1$.
Then, there is a unique solution $X^N_n$ 
of \eqref{eq:full-scheme}.
\end{lemma}

\section{Moment boundedness}\label{Sec:well-posedness}

This section is devoted to the moment estimates for the solutions of \eqref{SWE1} and  \eqref{eq:full-scheme}.

We first derive the moment bound in $\mathbb H^2$ provided the initial data $X_0=(u_0, v_0)^{\top} \in \mathbb H^2$. For any $(u,v)^{\top}\in \mathbb H^2$, denote
$$
 \mathcal H_2^\epsilon(u,v):=\| u\|_{\dot{H}^2}^2+\| v\|_{\dot{H}^1}^2 + \epsilon \big(\langle u,v\rangle_{\dot{H}^1}+\tfrac \eta2 \| u \|_{\dot{H}^1}^2 \big).
$$
Note that 
\begin{equation}\label{eq:H2bound}
    \tfrac{2-\epsilon }2  
 ( \| u \|_{\dot{H}^2}^2 + \| v \|_{\dot{H}^1}^2)
   \le \mathcal H_2^{\epsilon}(u,v) \le \tfrac{\epsilon + \epsilon\eta + 2}2  
 ( \| u \|_{\dot{H}^2}^2 + \| v\|_{\dot{H}^1}^2 ).
\end{equation}
\begin{proposition}\label{prop:bound-moment}
Let \cref{ass:nonlinearity,ass:noise} hold and
$X_0 \in \mathbb{H}^2$.
Then for all $p\ge 1$ 
and sufficiently small $\epsilon > 0$, it holds that
\begin{align}\label{eq:moment_bound-H2}
 \E \big[\mathcal H_2^{\epsilon}(u(t),v(t))^p\big]
    \leq C(p) \mathcal H_2^{\epsilon}(u_0,v_0)^p+C^\prime(p),
    \, \forall\; t \ge 0
\end{align}
for some positive constants $C(p)$ and $C^\prime(p)$. Moreover, we have
\begin{equation}\label{eq:average-X}
\sup_{ t \ge 0} 
   \E \big[ \| X(t) \|_{\mathbb{H}^2}^p \big] 
   \leq C(p) \big( 1 + \|X_0\|_{\mathbb{H}^2}^p \big).
\end{equation} 
\end{proposition}
\begin{proof}
We now proceed to prove \eqref{eq:moment_bound-H2} by induction on $p$. We start with the base case $p = 1$. A routine calculation gives
\begin{align*}
&\dd \mathcal H_2^{\epsilon} (u(t),v(t))\\
&\quad =\big( - 2 \eta \| v(t) \|_{\dot{H}^1}^2
  - 2 \langle v(t), F(v(t)) \rangle_{\dot{H}^1}
  + \big\| \Lambda^{\frac12} Q^{\frac12} \big\|_{\mathcal{L}_2}^2\big)\dd t+2\langle v(t),\dd W(t)\rangle_{\dot{H}^1}
  \\ 
  & \quad\quad\, + \epsilon \big( \| v(t) \|_{\dot{H}^1}^2 
     - \| u(t) \|_{\dot{H}^2}^2
     - \langle u(t), F(v(t)) \rangle_{\dot{H}^1} \big)\dd t+\epsilon\langle u(t),\dd W(t)\rangle_{\dot{H}^1}
  \\ 
  &\quad \leq \Big[-2(\eta + a_2 ) \| v(t) \|_{\dot{H}^1}^2
         - \tfrac{\epsilon}{2} \| u(t) \|_{\dot{H}^2}^2
         + \epsilon \big( \| v(t) \|_{\dot{H}^1}^2
         +  \| F(v(t)) \|^2 \big) 
         + \big\| \Lambda^{\frac12}  \big\|_{\mathcal{L}_2^0}^2 \Big]\dd t
  \\
  &\quad\quad\,+\langle 2v(t)+\epsilon u(t),\dd W(t)\rangle_{\dot{H}^1}
   \\ 
  &\quad \leq \Big[-2(\eta + a_2 ) \| v(t) \|_{\dot{H}^1}^2
         - \tfrac{\epsilon}{2} \| u(t) \|_{\dot{H}^2}^2
         + \epsilon ( 1 + 2a_1^2 )\lambda_1^{-2} \| v(t) \|_{\dot{H}^1}^2
         + 2 a_1^2 \epsilon
|\mathcal O|        + \big\| \Lambda^{\frac12} \big\|_{\mathcal{L}_2^0}^2 \Big]\dd t
  \\
  &\quad\quad\,+\langle 2v(t)+\epsilon u(t),\dd W(t)\rangle_{\dot{H}^1}.
\end{align*}
By choosing sufficiently small $\epsilon\in
(0,\;\min \{\tfrac{\eta+a_2}{1+2a_1^2}\lambda_1^2,2\})$, we obtain
\begin{align*}
\dd \mathcal H_2^{\epsilon} (u(t),v(t))
  &\leq \Big[-(\eta + a_2 ) \| v(t) \|_{\dot{H}^1}^2
         - \tfrac{\epsilon}{2} \| u(t) \|_{\dot{H}^2}^2
         + 2 a_1^2 \epsilon{|\mathcal O|}  
         + \big\| \Lambda^{\frac12}  \big\|_{\mathcal{L}_2^0}^2 \Big]\dd t
  \\
  &\quad\,+\langle 2v(t)+\epsilon u(t),\dd W(t)\rangle_{\dot{H}^1}.
\end{align*}
Owning to \eqref{eq:H2bound},
there exist two positive constants $C_1:=\frac{\epsilon}{\epsilon+\epsilon\eta+2}\min\{\frac{2(1+2a_1^2)}{\lambda_1^2},1\},C_2:= 2 a_1^2 \epsilon|\mathcal O|
         + \big\| \Lambda^{\frac12}  \big\|_{\mathcal{L}_2^0}^2$ such that
\begin{equation}\label{h2-estimate}
 \dd \mathcal H_2^{\epsilon} (u(t),v(t)) \leq \big(- C_1 \mathcal H_2^{\epsilon}(u(t),v(t)) + C_2\big)\dd t+\langle 2v(t)+\epsilon u(t),\dd W(t)\rangle_{\dot{H}^1}.
\end{equation}
Combining the Gronwall inequality yields
\begin{align*}
\E \big[ \mathcal H_2^{\epsilon} (u(t),v(t)) \big]
\leq e^{- C_1 t} \mathcal H_2^{\epsilon}(u_0,v_0) + C_2.
\end{align*}
This implies \eqref{eq:moment_bound-H2} for the base case $p=1$.

For $p\geq 2$, we assume that \eqref{eq:moment_bound-H2} holds for the cases up to $p-1$. It suffices to show that the case $p$ holds.
We have
\begin{align}\label{eq:H_2}
\begin{split}
\dd \mathcal H_2^{\epsilon}(u(t),v(t))^p 
   =&\, p \mathcal H_2^{\epsilon}(u(t),v(t))^{p-1}\dd \mathcal H_2^{\epsilon}(u(t),v(t))  \\
   &+\, \tfrac12 p (p-1) \mathcal H_2^{\epsilon}(u(t),v(t))^{p-2}
      \sum_{k \ge 1} q_k \big \vert 
      \langle  2 \Lambda v(t) + \epsilon\Lambda u(t), e_k \rangle \vert^2\dd t.
\end{split}
\end{align}
For the first term on the above right hand side,
by using \eqref{eq:H2bound}, \eqref{h2-estimate} and Young's inequality,
there exist two positive constants $C_3$ and $C_4$, depending on $p$, such that
\begin{align*}
&p \mathcal H_2^{\epsilon}(u(t),v(t))^{p-1}\dd \mathcal H_2^{\epsilon}(u(t),v(t))\\ 
 &\quad \leq \Big[- C_3 \mathcal{H}_2^{\epsilon}(u(t),v(t))^{p-1}
       ( \| u(t) \|_{\dot{H}^2}^2 + \| v (t)\|_{\dot{H}^1}^2 )
   \\ & \qquad\qquad    - C_3 \mathcal{H}_2^{\epsilon}(u(t),v(t))^{p} + C
       \mathcal{H}_2^{\epsilon}(u(t),v(t))^{p-1} \Big] \dd t
    \\ & \qquad\qquad + p \mathcal{H}_2^{\epsilon}(u(t),v(t))^{p-1}
    \langle 2v(t)+\epsilon u(t),\dd W(t)\rangle_{\dot{H}^1}
\\ &\quad \leq \Big[ -C_4 \mathcal{H}_2^{\epsilon}(u(t),v(t))^{p-2}
       ( \| u(t) \|_{\dot{H}^2}^2 + \| v (t)\|_{\dot{H}^1}^2 )^2
       - C_4 \mathcal{H}_2^{\epsilon}(u(t),v(t))^{p} + C \Big] \dd t
     \\ & \qquad \qquad  + p \mathcal{H}_2^{\epsilon}(u(t),v(t))^{p-1}
    \langle 2v(t)+\epsilon u(t),\dd W(t)\rangle_{\dot{H}^1}.
\end{align*}
For the second term on the right hand side of \eqref{eq:H_2}, by using the Young inequality, we derive
\begin{align*}
\tfrac12 & p (p-1) \mathcal{H}_2^{\epsilon}(u(t),v(t))^{p-2}
      \sum_{k \ge 1} q_k \big \vert 
      \langle  2 \Lambda v(t) + \epsilon\Lambda u(t), e_k \rangle \vert^2
\\ & \leq p (p-1) \mathcal{H}_2^{\epsilon}(u(t),v(t))^{p-2} 
\operatorname{Tr}(\Lambda Q)
   ( 4 \| v(t) \|_{\dot{H}^1}^2 +  \epsilon^2\| u(t) \|_{\dot{H}^2}^2 )
\\ & \leq \tfrac12 C_4 \mathcal{H}_2^{\epsilon}(u(t),v(t))^{p-2}
 (  \| v (t)\|_{\dot{H}^1}^2 +  \| u (t)\|_{\dot{H}^2}^2 )^2
 + C \mathcal{H}_2^{\epsilon}(u(t),v(t))^{p-2}
 \\ & \leq \tfrac12 C_4 \mathcal{H}_2^{\epsilon}(u(t),v(t))^{p-2}
 (  \| v (t)\|_{\dot{H}^1}^2 +  \| u (t)\|_{\dot{H}^2}^2 )^2
 + \tfrac 12 C_4\mathcal{H}_2^{\epsilon}(u(t),v(t))^{p}+C.
\end{align*}
Combining the above estimates leads to
\begin{equation*}
\begin{split}
    \dd  \mathcal H_2^{\epsilon}(u(t),v(t))^p 
      & \leq - C \mathcal H_2^{\epsilon}(u(t),v(t))^p \dd t + C 
       \\  &\quad + p \mathcal{H}_2^{\epsilon}(u(t),v(t))^{p-1}
    \langle 2v(t)+\epsilon u(t),\dd W(t)\rangle_{\dot{H}^1}.
\end{split}
\end{equation*}
Taking expectation and utilizing  Gronwall's inequality produces \eqref{eq:moment_bound-H2} for the general case $p \ge 2$.
The proof  of \eqref{eq:average-X} is thus completed by \eqref{eq:H2bound}.
\end{proof}

For any $(u,v)^{\top}\in \mathbb H^1$,  we denote
$$
 \mathcal H_1^\epsilon(u,v):=\| u\|_{\dot{H}^1}^2+\| v\|^2 + \epsilon \big(\langle u,v\rangle+\tfrac \eta2 \| u \|^2 \big).
$$
By following a similar approach to that in \cref{prop:bound-moment}, we have the moment boundedness in $\mathbb H^1$ provided the initial data $X_0=(u_0, v_0)^{\top} \in \mathbb H^1$.
The proof is omitted.
\begin{proposition}\label{prop:H1-moment}
Let \cref{ass:nonlinearity,ass:noise} hold and
$X_0 \in \mathbb{H}^1$.
Then for all $p\ge 1$ and sufficiently small $\epsilon > 0$, it holds that
\begin{align}\label{eq:moment_bound-H1}
 \E \big[\mathcal H_1^{\epsilon}(u(t),v(t))^p\big]
    \leq C(p) \mathcal H_1^{\epsilon}(u_0,v_0)^p
    +C^\prime(p),
    \, \forall\; t \ge 0
\end{align}
for some positive constants $C(p)$ and $C^\prime(p)$. Moreover, we have
\begin{equation}
\sup_{ t \ge 0}   \E \big[ \| X(t) \|_{\mathbb{H}^1}^p  \big] 
   \leq C(p) \big( 1+\|X_0\|_{\mathbb{H}^1}^p \big).
\end{equation} 
\end{proposition}

For the spatial semi-discretization, we also have moment boundedness in ${\mathbb H}^2$ of the numerical solution, whose proof is similar to that of \cref{prop:bound-moment} and thus is omitted.
\begin{lemma}
Let \cref{ass:nonlinearity,ass:noise} hold and
$X_0 \in \mathbb{H}^2$.
Then, for any $t > 0$, there exists a  positive constant $C(p)$ such that
\begin{equation*}
\E \big[  \| X^N(t) \|_{\mathbb{H}^2}^{p} \big]
     \dd s \leq C(p)(1+\|X_0\|^p_{{\mathbb H}^2}), \, t \ge 0.
\end{equation*}
\end{lemma}

\begin{lemma}
Let \cref{ass:nonlinearity,ass:noise} hold and
$X_0 \in \mathbb{H}^2$.
For any $\tau \in (0,1)$,
there exists a  positive constant $C$ such that
\begin{equation}\label{eq:discrete-sum-X}
\sup_{m \in \N} \frac 1m \sum_{k=0}^m
   \E \big[ \| X_k^N \|_{\mathbb{H}^2}^2 \big]
     \leq  C \big( 1 +  \| X_0 \|_{\mathbb{H}^2}^2  \big).
\end{equation}
\end{lemma}

\begin{proof}
We  first give the estimate of the time average of the second moment for the component $v_k^N$. Note that 
\begin{align}\label{eq:vk-1}
\begin{split}
\langle & u_{n+1}^N  - u_n^N, u_{n+1}^N \rangle_{\dot{H}^2}
 + \langle v_{n+1}^N - v_n^N, v_{n+1}^N \rangle_{\dot{H}^1}
  \\ & = \tfrac 12 \big( \| u_{n+1}^N \|_{\dot{H}^2}^2
            - \| u_{n}^N \|_{\dot{H}^2}^2
              + \| u_{n+1}^N - u_n^N \|_{\dot{H}^2}^2 \big)
  \\ & \quad +\tfrac 12 \big( \| v_{n+1}^N \|_{\dot{H}^1}^2
            - \| v_{n}^N\|_{\dot{H}^1}^2
              + \| v_{n+1}^N - v_n^N \|_{\dot{H}^1}^2 \big)
  \\ & \geq \tfrac 12 \big( \| u_{n+1}^N \|_{\dot{H}^2}^2
            - \| u_{n}^N \|_{\dot{H}^2}^2
            + \| v_{n+1}^N \|_{\dot{H}^1}^2
            - \| v_{n}^N \|_{\dot{H}^1}^2 \big)
         +  \tfrac 12 \| v_{n+1}^N - v_n^N \|_{\dot{H}^1}^2.
\end{split}
\end{align}
It follows from \eqref{eq:full-scheme} that
\begin{align}\label{eq:vk-2}
\begin{split}
& \langle  u_{n+1}^N  - u_n^N, u_{n+1}^N \rangle_{\dot{H}^2}
 + \langle v_{n+1}^N - v_n^N, v_{n+1}^N \rangle_{\dot{H}^1}
  \\ & = \tau \langle v_{n+1}^N, u_{n+1}^N \rangle_{\dot{H}^2}
         - \tau \langle \Lambda_N u_{n+1}^N +\eta v_{n+1}^N +  \Pi_N F(v_{n+1}^N),
            v_{n+1}^N \rangle_{\dot{H}^1}
         + \langle  \delta W_n, v_{n+1}^N \rangle_{\dot{H}^1}
  \\ & = - \tau\eta \| v_{n+1}^N \|_{\dot{H}^1}^2
         - \tau \langle   F(v_{n+1}^N), v_{n+1}^N \rangle_{\dot{H}^1}
         + \langle  \delta W_n, v_{n+1}^N - v_n^N \rangle_{\dot{H}^1}
         + \langle  \delta W_n, v_{n}^N \rangle_{\dot{H}^1}
  \\ & \leq - ( \eta + a_2) \tau \| v_{n+1}^N \|_{\dot{H}^1}^2
         + \tfrac 12  \| v_{n+1}^N - v_n^N \|_{\dot{H}^1}^2
         + \tfrac 12  \|  \delta W_n \|_{\dot{H}^1}^2
         + \langle  \delta W_n, v_{n}^N \rangle_{\dot{H}^1},
\end{split}
\end{align}
where we used
\begin{equation*}
\langle  F(v_{n+1}^N), v_{n+1}^N \rangle_{\dot{H}^1}
  \ge a_2 \| v_{n+1}^N \|_{\dot{H}^1}^2.
\end{equation*}
Combining \eqref{eq:vk-1} and \eqref{eq:vk-2} and taking expectation on both sides then yield
\begin{equation*}
\begin{split}
\E \big[ & \| u_{n+1}^N \|_{\dot{H}^2}^2
        - \| u_{n}^N \|_{\dot{H}^2}^2
         +  \| v_{n+1}^N \|_{\dot{H}^1}^2
           - \| v_{n}^N \|_{\dot{H}^1}^2 \big]
  \\ & \leq - 2( \eta + a_2) \tau \E [ \| v_{n+1}^N \|_{\dot{H}^1}^2 ]
    + \tau \| \Lambda^{\frac12} \|_{\mathcal{L}_2^0}^2,
\end{split}
\end{equation*}
which leads to
\begin{equation}\label{eq:lem:moment-sum-v}
\begin{split}
\E \big[ \| u_{n+1}^N \|_{\dot{H}^2}^2
    &  +  \| v_{n+1}^N \|_{\dot{H}^1}^2 \big]
   + 2( \eta + a_2) \tau \sum_{k=1}^{n+1}
   \E [ \| v_{k}^N \|_{\dot{H}^1}^2 ]
  \\ & \leq   \| u_0 \|_{\dot{H}^2}^2
               +  \| v_0 \|_{\dot{H}^1}^2 
      + (n+1) \tau \| \Lambda^{\frac12} \|_{\mathcal{L}_2^0}^2.
\end{split}
\end{equation}
To estimate the time average of the second moment for the component $u_k^N$, we
 note that
\begin{align*}
\langle  \eta u_{n+1}^N &  + v_{n+1}^N, u_{n+1}^N - u_n^N \rangle_{\dot{H}^1}
 + \langle \Lambda u_{n+1}^N, v_{n+1}^N - v_n^N \rangle
 \\ & = \langle  \eta u_{n+1}^N, u_{n+1}^N - u_n^N \rangle_{\dot{H}^1}
       + \langle u_{n+1}^N, v_{n+1}^N \rangle_{\dot{H}^1}
       - \langle u_{n}^N, v_{n}^N \rangle_{\dot{H}^1}
 \\ & \quad  + \tau \langle v_{n+1}^N, v_{n+1}^N - v_n^N \rangle_{\dot{H}^1}
 \\ & = \langle \eta u_{n+1}^N, u_{n+1}^N - u_n^N \rangle_{\dot{H}^1}
       + \langle u_{n+1}^N, v_{n+1}^N \rangle_{\dot{H}^1}
       - \langle u_{n}^N, v_{n}^N \rangle_{\dot{H}^1}
 \\ & \quad  + \tau \| v_{n+1}^N \|_{\dot{H}^1}^2
       - \tau \langle v_{n+1}^N, v_n^N \rangle_{\dot{H}^1}
 \\ & \geq \tfrac \eta2 \| u_{n+1}^N \|_{\dot{H}^1}^2
       - \tfrac \eta2 \| u_{n}^N \|_{\dot{H}^1}^2
       + \tfrac \eta2 \| u_{n+1}^N - u_n^N \|_{\dot{H}^1}^2
  \\ & \quad  + \langle u_{n+1}^N, v_{n+1}^N \rangle_{\dot{H}^1}
       - \langle u_{n}^N, v_{n}^N \rangle_{\dot{H}^1}
       + \tfrac12 \tau \| v_{n+1}^N \|_{\dot{H}^1}^2
       - \tfrac12 \tau \| v_{n}^N \|_{\dot{H}^1}^2.
\end{align*}
On the other hand, using \cref{ass:nonlinearity} gives
\begin{align*}
\langle \eta u_{n+1}^N & + v_{n+1}^N, u_{n+1}^N - u_n^N \rangle_{\dot{H}^1}
 + \langle \Lambda u_{n+1}^N, v_{n+1}^N - v_n^N \rangle
  \\ & = \tau \langle \eta \Lambda u_{n+1}^N + \Lambda v_{n+1}^N,
     v_{n+1}^N \rangle  + \langle \delta W_n, \Lambda u_{n+1}^N \rangle
  \\ & \quad   - \tau \langle \Lambda u_{n+1}^N,
          \Lambda u_{n+1}^N + \eta v_{n+1}^N +\Pi_N F(v_{n+1}^N) \rangle
  \\ & = \tau \| v_{n+1}^N \|_{\dot{H}^1}^2
        - \tau \| u_{n+1}^N \|_{\dot{H}^2}^2
     - \tau \langle \Lambda u_{n+1}^N,  \Pi_N F (v_{n+1}^N) \rangle
  \\ & \quad  + \langle \delta W_n,
         u_{n+1}^N - u_n^N \rangle_{\dot{H}^1}
     + \langle \delta W_n, u_n^N \rangle_{\dot{H}^1}
  \\ & \leq \tau \| v_{n+1}^N \|_{\dot{H}^1}^2
          - \tau \| u_{n+1}^N \|_{\dot{H}^2}^2
          + \tfrac12 \tau \| u_{n+1}^N \|_{\dot{H}^2}^2
          +  a_1^2 \tau  ( \vert \mathcal{O} \vert
          + \| v_{n+1}^N \|^2 )
  \\  & \quad + \tfrac \eta2 \| u_{n+1}^N - u_n^N \|_{\dot{H}^1}^2
     + \tfrac1{2\eta} \| \delta W_n\|_{\dot{H}^1}^2
     + \langle \delta W_n, u_n^N \rangle_{\dot{H}^1}
  \\ & = ( 1 + a_1^2) \tau  \| v_{n+1}^N \|_{\dot{H}^1}^{2}
          - \tfrac12 \tau \| u_{n+1}^N \|_{\dot{H}^2}^2
          + \tfrac \eta2 \| u_{n+1}^N - u_n^N \|_{\dot{H}^1}^2
 \\ & \quad  + \tfrac1{2\eta} \| \delta W_n\|_{\dot{H}^1}^2
          + \langle \delta W_n, u_n^N \rangle_{\dot{H}^1}
          + a_1^2 \vert \mathcal{O} \vert \tau.
\end{align*}
It follows by taking expectation and using \eqref{eq:lem:moment-sum-v} that,
\begin{align*}
\E  \big[ & \tfrac \eta2 \| u_{n+1}^N \|_{\dot{H}^1}^2
           + \tfrac12 \tau \| v_{n+1}^N \|_{\dot{H}^1}^2
           + \langle u_{n+1}^N, v_{n+1}^N \rangle_{\dot{H}^1} \big]
    + \tfrac \tau2 \sum_{k=1}^{n+1} \E [ \| u_k^N \|_{\dot{H}^2}^2 ]
  \\ & \leq  ( 1 + a_1^2) \tau \sum_{k=1}^{n+1}
         \E [ \| v_k^N \|_{\dot{H}^1}^{2} ]
      + \tfrac \eta2  \| u_{0} \|_{\dot{H}^1}^2
           + \tfrac12 \tau  \| v_{0} \|_{\dot{H}^1}^2
           +  \langle u_{0}, v_{0} \rangle_{\dot{H}^1}  
  \\ & \quad    + C \cdot ( n + 1 ) \tau
       ( 1 +  \| \Lambda^{\frac12} \|_{\mathcal{L}_2^0}^2 )
  \\ &  \leq C  \big(  \| X_0 \|_{\mathbb{H}^2}^2 
       + ( n + 1 ) \tau \big).
\end{align*}
Hence,
\begin{align*}
\tfrac \tau2 \sum_{k=1}^{n+1} \E [ \| u_k^N \|_{\dot{H}^2}^2 ]
     &  \leq \E  \big[  \big \vert \langle u_{n+1}^N,
      v_{n+1}^N \rangle_{\dot{H}^1} \big \vert \big]
     +  C  \big(  \| X_0 \|_{\mathbb{H}^2}^2 
       + ( n + 1 ) \tau \big)
  \\ & \leq \E \big[ \| X_{n+1}^N \|_{\mathbb H^2}^2 \big]
     +  C  \big(  \| X_0 \|_{\mathbb{H}^2}^2
       + ( n + 1 ) \tau \big).
\end{align*}
Turning back to \eqref{eq:lem:moment-sum-v}
leads to
\begin{equation*}
\tau \sum_{k=1}^{n+1} \E \big[ \| u_k^N \|_{\dot{H}^2}^2  \big]
  \leq C \big( \| X_0 \|_{\mathbb{H}^2}^2
       + ( n + 1 ) \tau \big).
\end{equation*}
Combining it with \eqref{eq:lem:moment-sum-v} finishes the proof.
\end{proof}

\section{Exponential ergodicity}\label{Sec:ergodicity}

In this section, we present the exponential ergodicity of solutions for  \eqref{SWE1} and \eqref{eq:full-scheme}.

\subsection{Exponential ergodicity of exact solution}
We introduce the Markov semigroup
$P_t \in \mathcal{L} ( \mathcal B_b (\mathbb{H}^1) )$ associated to \eqref{SWE1}, which is defined by
\begin{equation*}
P_t \varphi (X_0) = \E \big[ \varphi ( X(t;X_0) ) \big], \,
X_0 \in \mathbb{H}^1,\, \varphi \in \mathcal B_b (\mathbb{H}^1), \, t \ge 0,
\end{equation*}
where $\mathcal B_b (\mathbb{H}^1)$  represents the Banach space of all Borel bounded mappings $\varphi: \mathbb{H}^1 \to \mathbb{R}$ and 
 $X(t;X_0)$ is the solution of \eqref{SWE1} with  initial datum $X_0$.
Denote by $\mathcal{P}(\mathbb{H}^1)$  the space of all probability measures on $(\mathbb{H}^1, \mathcal{B}(\mathbb{H}^1))$ and by $\mathcal{P}_p(\mathbb{H}^1)$ the  space of all probability measures on $(\mathbb{H}^1, \mathcal{B}(\mathbb{H}^1))$ with finite $p$th moment. The dual operator of $P_t$ is defined as
\begin{equation*}
\nu P_t (A) = \int_{\mathbb{H}^1} P_t \mathds{1}_A (U) \nu (\dd U), \,
A \in \mathcal{B}(\mathbb{H}^1), \, \nu \in \mathcal{P}(\mathbb{H}^1).
\end{equation*}

The following proposition states that \eqref{SWE1} admits an invariant measure. 

\begin{proposition}\label{prop:existence-IM-exact}
Let \cref{ass:nonlinearity,ass:noise} hold.
Then there exists an invariant measure $\mu\in\mathcal P(\mathbb H^1)$ for $P_t$.
\end{proposition}

\begin{proof}
We first show that the Markov semigroup $P_t$ is Feller.
Setting $\bar{u}(t) = u(t) - \tilde{u}(t)$ and
$\bar{v}(t) = v(t) - \tilde{v}(t)$,
we have
\begin{align*}
\begin{cases}
\frac{\dd}{\dd t} \bar{u}(t)  = \bar{v}(t),\\
\frac{\dd}{\dd t} \bar{v}(t)  = - \Lambda \bar{u}(t) - \eta\bar{v}(t)
    - \big[ F( v(t) ) - F( \tilde{v}(t) ) \big],\\
    \bar{u}(0) = \bar{u}_0 = u_0 - \tilde{u}_0,\,\,
    \bar{v}(0) = \bar{v}_0 = v_0 - \tilde{v}_0.
\end{cases}
\end{align*}
It follows from \eqref{oneside-lip} that
\begin{equation*}
\begin{split}
\tfrac{\dd}{\dd t} & \| X(t) - \widetilde{X}(t) \|_{\mathbb{H}^1}^2
   = \tfrac{\dd}{\dd t}
    \big( \| \bar{u}(t)  \|_{\dot{H}^1}^2 + \| \bar{v}(t)  \|^2 \big)
  \\ & = 2 \langle \bar{u}(t), \bar{v}(t) \rangle_{\dot{H}^1}
           - 2 \left\langle \bar{v}(t), \Lambda \bar{u}(t) + \eta \bar{v}(t)
    +  F( v(t) ) - F( \tilde{v}(t) )  \right\rangle
  \\ & = - 2\eta \| \bar{v}(t) \|^2 - 2 \left\langle \bar{v}(t),
     F( v(t) ) - F( \tilde{v}(t) )  \right\rangle
  \\ & \leq - 2 (\eta + a_2 ) \| \bar{v}(t) \|^2
   \leq 0.
\end{split}
\end{equation*}
Hence, for any $t \ge 0$,
\begin{equation*}
\| X(t) - \widetilde{X}(t) \|_{\mathbb{H}^1} \leq
   \| X_0 - \widetilde{X}_0 \|_{\mathbb{H}^1},
\end{equation*}
which yields the Feller property of $P_t$.

Fix $X_0\in \mathbb H^2$. By the Krylov--Bogoliubov theorem (see e.g., \cite[Theorem 7.1]{Prato2006introduction}), it  suffices to prove the tightness of  the family of probability measures
$\{ \mu_T^{X_0} \}_{T > 0}$ on $\mathcal{B}(\mathbb{H}^1)$,
given by
\begin{equation*}
\mu_T^{X_0} := \frac 1T \int_0^T \delta_{X_0}P_t \dd t, \, T > 0.
\end{equation*}
For  any $R > 0$, define
$
K_R = \left\{ y \in \mathbb{H}^1:
   \| y \|_{\mathbb{H}^2} \leq R \right\},
$
which  is a compact subset in $\mathbb{H}^1$ since the embedding $\mathbb{H}^2 \subset \mathbb{H}^1$ is compact.
By using Chebyshev's inequality and \eqref{eq:average-X},
we have
\begin{equation*}
\begin{split}
\mu_T^{X_0} (K_R^c) & = \frac 1T \int_0^T \mathbb{P}
  \big( \| X(t;X_0) \|_{\mathbb{H}^2} > R \big) \dd t
   \\ & \leq \frac{1}{T R^2}  \int_0^T
            \E \big[ \| X(t;X_0) \|_{\mathbb{H}^2}^2 \big] \dd t
    \leq \frac{C_{X_0}}{R^2}.
\end{split}
\end{equation*}
Hence, for any $\epsilon > 0$, there exists $R_\epsilon > 0$ such that
$\mu_T^{X_0} (K_{R_\epsilon}^c) < \epsilon$.
This shows that $\{ \mu_T^{X_0} \}_{T > 0}$ is tight.
The proof is thus completed.
\end{proof}

In the sequel, to ensure the uniqueness of the invariant measure $\mu$ for \eqref{SWE},
we need the following assumption on  $F$.
\begin{assumption}\label{ass:nonlinearity2}
There exists a positive constant $L$ such that
\begin{equation*}
\| F(v_1) - F (v_2) \|_{\dot{H}^{-1}}
 \leq L \|v_1-v_2\|,\,\, v_1,v_2 \in H.
\end{equation*}
\end{assumption}

\begin{lemma}\label{lem:exp-convergence}
Under \cref{ass:nonlinearity,ass:noise,ass:nonlinearity2} and $X_0,\widetilde{X}_0 \in \mathbb{H}^1$  with finite $p$th moment. Let $X(t)$ and $\widetilde{X}(t)$ be the solutions of \eqref{SWE1} with different initial data $X_0$ and $\widetilde{X}_0$, respectively.
Then for $\epsilon > 0$ sufficiently small  and any $ p>0, t \ge 0$,
\begin{equation}\label{eq:difference-continuous-as}
 \| X(t) - \widetilde{X}(t) \|_{\mathbb{H}^1}
   \leq \tfrac{1+2\epsilon}{1-2\epsilon} ~ e^{- \epsilon t}
       \| X_0 - \widetilde{X}_0 \|_{\mathbb{H}^1}, \text{ a.s.}
\end{equation}
and
\begin{equation}\label{eq:difference-continuous-Lp}
\| X(t) - \widetilde{X}(t) \|_{L^p(\Omega;\mathbb{H}^1)}
   \leq \tfrac{1+2\epsilon}{1-2\epsilon} ~  e^{- \epsilon t}
       \| X_0 - \widetilde{X}_0 \|_{L^p(\Omega;\mathbb{H}^1)}.
\end{equation}
\end{lemma}

\begin{proof}
Note that
\begin{equation*}
\begin{split}
\tfrac{\dd}{\dd t} \langle \bar{u}(t), \bar{v}(t) \rangle
  & = \langle \bar{v}(t), \bar{v}(t) \rangle
     - \langle \bar{u}(t),
        \Lambda \bar{u}(t) + \eta \bar{v}(t)
         + F(v(t)) - F(\tilde{v}(t)) \rangle
 \\ & = \| \bar{v}(t) \|^2 - \| \bar{u}(t) \|_{\dot{H}^1}^2
        - \eta  \langle \bar{u}(t), \bar{v}(t) \rangle
        - \langle \bar{u}(t), F(v(t)) - F(\tilde{v}(t)) \rangle
 \\ & \leq \| \bar{v}(t) \|^2 - \| \bar{u}(t) \|_{\dot{H}^1}^2
        - \eta  \langle \bar{u}(t), \bar{v}(t) \rangle
        + L \| \bar{u}(t)\|_{\dot{H}^1} \| \bar{v}(t) \|,
\end{split}
\end{equation*}
where \cref{ass:nonlinearity2} was used in the last step.
As a result, we obtain
\begin{equation*}
\begin{split}
\tfrac{\dd}{\dd t} & \left( \| \bar{u}(t) \|_{\dot{H}^1}^2
   + \| \bar{v}(t) \|^2
   + 4 \epsilon \, \langle \bar{u}(t), \bar{v}(t) \rangle \right)
  \\ & \leq - \big[ 2 ( \eta  + a_2) -
       4 \epsilon \big] \| \bar{v}(t) \|^2
         -4 \epsilon \, \| \bar{u}(t) \|_{\dot{H}^1}^2
           -4 \epsilon \,\eta \, \langle \bar{u}(t), \bar{v}(t) \rangle
            + 4 \epsilon \, L \, \| \bar{u}(t)\|_{\dot{H}^1}
                \| \bar{v}(t) \|
 \\ & \leq  - \big[ 2 ( \eta  + a_2) -
       4 ( L^2 + \eta^2 +1 ) \epsilon \big] \| \bar{v}(t) \|^2
            -  2 \epsilon \, \| \bar{u}(t) \|_{\dot{H}^1}^2
 \\ & \leq  - ( \eta  + a_2) \| \bar{v}(t) \|^2
            -  2 \epsilon \, \| \bar{u}(t) \|_{\dot{H}^1}^2,
\end{split}
\end{equation*}
where we used  Young's inequality and
$ 0 < \epsilon < \tfrac{\eta  + a_2}{4 ( L^2 + \eta^2 +1 )}$.
Furthermore, we get
\begin{equation*}
\tfrac{\dd}{\dd t}  \left( \| \bar{u}(t) \|_{\dot{H}^1}^2
   + \| \bar{v}(t) \|^2
   + 4 \epsilon \, \langle \bar{u}(t), \bar{v}(t) \rangle \right)
 \leq  -  \epsilon \left( \| \bar{u}(t) \|_{\dot{H}^1}^2
   + \| \bar{v}(t) \|^2  + 4 \epsilon \, \langle \bar{u}(t),
       \bar{v}(t) \rangle \right),
\end{equation*}
for sufficiently small $\epsilon > 0$.
Thus, we invoke Gronwall's inequality to deduce
\begin{equation*}
 \|  \bar{u}(t) \|_{\dot{H}^1}^2
      + \| \bar{v}(t) \|^2  + 4 \epsilon \, \langle \bar{u}(t),
       \bar{v}(t) \rangle
 \leq e^{- \epsilon t}
         \left( \| \bar{u}_0 \|_{\dot{H}^1}^2
        + \| \bar{v}_0 \|^2  + 4 \epsilon \, \langle \bar{u}_0,
           \bar{v}_0 \rangle \right).
\end{equation*}
Applying the fact
$-\tfrac12 ( a^2 + b^2 ) \leq ab \leq \tfrac12 ( a^2 + b^2 )$ yields
\begin{equation*}
 \|  \bar{u}(t) \|_{\dot{H}^1}^2 + \| \bar{v}(t) \|^2
    \leq \tfrac{1+2\epsilon}{1-2\epsilon} ~ e^{- \epsilon t}
         \left( \| \bar{u}_0 \|_{\dot{H}^1}^2
        + \| \bar{v}_0 \|^2  \right)
\end{equation*}
Finally, taking expectation, we obtain \eqref{eq:difference-continuous-Lp}.
\end{proof}

The definition of $p$th order Wasserstein distance indicates that
$$
\mathcal{W}_p\left(\nu_1, \nu_2\right)=
 \inf \{ \|Y_1- Y_2\|_{L^p(\Omega;\mathbb H^1)}:~Y_1 \sim \nu_1 \text{ and } Y_2 \sim \nu_2\}.
$$
 Since $\nu_1 P_t$ (resp. $\nu_2 P_t$) is the distribution of solution $X(t;X_0)$ (resp. ${X}(t;\widetilde{X}_0)$) with initial distribution $X_0\sim \nu_1$ (resp. $\widetilde{X}_0\sim\nu_2$),
 one can  invoke \eqref{eq:difference-continuous-Lp} to obtain
$$
\mathcal{W}_p\left(\nu_1 P_t, \nu_2 P_t\right)\le \| X(t;X_0) - {X}(t;\widetilde{X}_0) \|_{L^p(\Omega;\mathbb{H}^1)}
   \leq \tfrac{1+2\epsilon}{1-2\epsilon} ~ e^{- \epsilon t}
       \| X_0 - \widetilde{X}_0 \|_{L^p(\Omega;\mathbb{H}^1)}.
$$
 Taking infimum w.r.t. $X_0$ and $\widetilde{X}_0$ gives that for any $\nu_1,\nu_2\in\mathcal P_p(\mathbb H^1)$,
\begin{align}\label{eq:exponential-convergence}
\mathcal W_p(\nu_1 P_t,\nu_2 P_t)\le \tfrac{1+2\epsilon}{1-2\epsilon} ~ e^{- \epsilon t} \mathcal W_p(\nu_1,\nu_2).
\end{align}

\begin{theorem}
Let \cref{ass:nonlinearity,ass:noise,ass:nonlinearity2} hold.
Then, \eqref{SWE1} admits a unique invariant measure $\mu\in \cap_{p>0}\mathcal P_p(\mathbb{H}^1)$ such that for any $t \ge 0,\; p>0,\; \nu \in \mathcal P_p(\mathbb{H}^1)$, and $\epsilon>0$ sufficiently small,
\begin{align}
\mathcal W_p(\nu P_t,\mu) \le \tfrac{1+2\epsilon}{1-2\epsilon} ~
e^{-\epsilon t}\mathcal W_p(\nu,\mu).
\end{align}
\end{theorem}

\begin{proof}
The exponential contraction property \eqref{eq:exponential-convergence}  implies that for any $\nu\in \mathcal P_p(\mathbb H^1)$ and
$s > 0,$
\begin{align*}
\mathcal W_p(\nu P_{t+s},\nu P_t)\le \tfrac{1+2\epsilon}{1-2\epsilon} ~ e^{- \epsilon t} \mathcal W_p(\nu P_s,\nu),
\end{align*}
which together with the completeness of $(\mathcal P_p(\mathbb H^1), \mathcal W_p)$ tells that $\{\nu P_t\}_{t\ge0}$ converges in $p$th order
Wasserstein distance. Denote by $\mu_\nu\in \mathcal P_p\left(\mathbb H^1\right)$ the limit of $\{ \nu P_t\}_{t\ge0}$ in $\mathcal W_p$, then for any $\nu_1, \nu_2 \in \mathcal P_p\left(\mathbb H^1\right)$,
$$
\mathcal W_p \left(\mu_{\nu_1}, \mu_{\nu_2}\right)
=\lim _{t \rightarrow \infty} \mathcal W_p
\left(\nu_1 P_t, \nu_2 P_t\right)
\leq \lim _{t \rightarrow \infty}
\tfrac{1+2\epsilon}{1-2\epsilon} ~ e^{-\epsilon t}
\mathcal W_p\left(\nu_1, \nu_2\right)=0.
$$
Hence, for any $\nu \in \mathcal P_p(\mathbb H^1)$, the limit $\mu_\nu\in \mathcal P_p\left(\mathbb H^1\right)$ of $\{ \nu P_t\}_{t\ge0}$  is independent of $\nu$. Then we denote the limit by $\mu$. Moreover, we have
$$
\mathcal W_p \left(\mu,  \mu P_t\right)=\lim _{s \rightarrow \infty} \mathcal W_p \left(\mu P_s, (\mu P_t) P_s\right)
\leq \lim_{s \rightarrow \infty}
\tfrac{1+2\epsilon}{1-2\epsilon} ~ e^{-\epsilon s}
\mathcal W_p\left(\mu, \mu P_t \right)=0,
$$
which means that $\mu$ is the unique invariant measure of $\left\{P_t\right\}_{t \geq 0}$ and $\mu \in \cap_{p>0} \mathcal P_p\left(\mathbb H^1\right)$.
Furthermore, for any $\nu\in\mathcal P_p(\mathbb H^1)$,   taking advantage of \eqref{eq:exponential-convergence}, one obtains
$$
\mathcal W_p(\nu P_t,\mu)
    =  \mathcal W_p(\nu P_t,\mu P_t)
     \le \tfrac{1+2\epsilon}{1-2\epsilon} ~ e^{-\epsilon t}
        \mathcal W_p(\nu, \mu),
$$
which finishes the proof.
\end{proof}

To investigate the exponential ergodicity with more general test functionals,   for $p>0,\gamma\in(0,1]$, we define $\mathcal{C}_{p, \gamma}:=$ $\mathcal{C}_{p, \gamma}(\B ; \R)$ the set of continuous functions $\varphi$ on the Banach space $\B$ such that
$$
\|\varphi\|_{p, \gamma}:=\sup _{\substack{x_1, x_2 \in \B \\ x_1 \neq x_2}} \frac{\left|\varphi(x_1)-\varphi(x_2)\right|}{d_{p, \gamma}(x_1, x_2)}<\infty
$$
with $
d_{p, \gamma}\left(x_1, x_2\right):=\left\|x_1-x_2\right\|_{\B}^\gamma\left(1+\left\|x_1\right\|_{\B}^p+\left\|x_2\right\|_{\B}^p\right)^{\frac{1}{2}}.$ If the context of the functional’s space is clear, we will abbreviate the notation of the space as
$\mathcal{C}_{p, \gamma}$. It follows from the definition that the functionals in $\mathcal{C}_{p, \gamma}(\B ; \R)$ have polynomial growth. Indeed, for any $\varphi\in\mathcal{C}_{p, \gamma}(\B ; \R)$, by means of the triangle inequality, one has
$$
|\varphi(x)|\le|\varphi(x)-\varphi(0)|+|\varphi(0)|\le \|\varphi\|_{p, \gamma}\|x\|_\B^\gamma(1+\|x\|_\B^p)^{\frac12}+|\varphi(0)|\le C_1(1+\|x\|_\B^{\gamma+\frac p2}).
$$
It is known that for all $t \geq 0, q \geq 1, P_t$ is uniquely extendible to a linear bounded operator on $L^q(\mathbb H^1, \mu)$ that we still denote by $P_t$. Since $\mu\in \cap_{q\ge 1}\mathcal P_q(\mathbb H^1)$, we know that $\mathcal C_{p,\gamma}(\mathbb H^1;\R)\subset \cap_{q\ge 1} L^{q}(\mathbb H^1,\mu)$. 
Hence $P_t\varphi$ is well-defined  for any $\varphi\in\mathcal C_{p,\gamma}$.

\begin{theorem}\label{thm:exp-ergodic}
Let \cref{ass:nonlinearity,ass:noise,ass:nonlinearity2} hold.
Then for any $t \ge 0, U \in \mathbb{H}^1$,
$\varphi\in \mathcal C_{p,\gamma}$, and sufficiently small $\epsilon>0$,
there exists a constant $C(\epsilon,p,\gamma)>0$ such that the unique invariant measure $\mu\in\mathcal P(\mathbb{H}^1)$ of \eqref{SWE1} satisfies
\begin{equation}\label{eq:exponential-ergodicity}
\left| P_t \varphi (U)
     -  \int_{\mathbb H^1} \varphi \dd \mu \right|
     \leq C(\epsilon,p,\gamma) \, \|\varphi\|_{p,\gamma}
     \, ( 1 + \| U \|_{\mathbb{H}^1}^{\gamma+\frac p2}  )
     e^{- \epsilon \gamma t}.
\end{equation}
\end{theorem}

\begin{proof}
For any $U,\widetilde{U} \in \mathbb{H}^1$, by using \eqref{eq:difference-continuous-Lp} and \cref{prop:H1-moment},
one obtains
\begin{align*}
  & \, \left| P_t\varphi(U) - P_t \varphi(\widetilde{U}) \right|
     =\,\left| \E [ \varphi( X(t;U) )
          - \varphi( X(t;\widetilde{U}) ) ] \right|
   \\ \le & \, \|\varphi\|_{p,\gamma}
    \E \left[ \left \| X(t;U) - X(t;\widetilde{U})
               \right\|_{\mathbb H^1}^\gamma
     \left( 1 +  \left \| X(t;U) \right\|_{\mathbb H^1}^p
         +  \| X(t;\widetilde{U})
         \|_{\mathbb H^1}^p\right)^{\frac{1}{2}} \right]
    \\ \le & \, \|\varphi\|_{p,\gamma} \left(
         \E \left[ \left \| X(t;U) - X(t;\widetilde{U})
       \right\|_{\mathbb H^1}^{2\gamma} \right] \right)^{\frac{1}{2}}
      \left( \E \left[ 1 + \left\| X(t;U) \right\|_{\mathbb H^1}^p
          +   \| X(t;\widetilde{U}) \|_{\mathbb H^1}^p
          \right] \right)^{\frac{1}{2}}
      \\ \le & \, C(\epsilon,p) \| \varphi \|_{p,\gamma}
              e^{-\epsilon\gamma t} \| U -\widetilde{U}
                     \|_{\mathbb H^1}^\gamma
          \left( 1 + \left \| U \right\|_{\mathbb H^1}^p
             +  \| \widetilde{U} \|_{\mathbb H^1}^p
  \right)^{\frac{1}{2}}.
\end{align*}
Hence, it follows from the above inequality that
\begin{align*}
   & \, \left| P_t \varphi(U) - \mu(\varphi) \right|
       =\,\left| \int_{\mathbb H^1} \big(P_t \varphi(U)
             - P_t \varphi(\widetilde{U})\big)
                \mu(\dd \widetilde{U}) \right| \notag
     \\ \le & \, C(\epsilon,p) \| \varphi \|_{p,\gamma}
      e^{-\epsilon\gamma t}  \int_{\mathbb H^1}
       \| U - \widetilde{U}\|_{\mathbb H^1}^\gamma
         \left( 1 + \left\| U \right\|_{\mathbb H^1}^p
             + \| \widetilde{U} \|_{\mathbb H^1}^p
      \right)^{\frac{1}{2}}
                  \mu(\dd \widetilde{U}) \notag
      \\ \le & \, C(\epsilon,p,\gamma) \| \varphi \|_{p,\gamma}
           e^{-\epsilon\gamma t}
             ( 1 + \| U \|_{\mathbb H^1}^{\gamma+\frac p2}),
\end{align*}
where we used $\mu\in \cap_{q\ge 1}\mathcal P_q(\mathbb H^1)$.
\end{proof}

\subsection{Exponential ergodicity of the numerical solution}\label{numerical-ergodicity}

This subsection shows the unique invariant measure and exponential ergodicity of the numerical solution for the full discretization \eqref{eq:full-scheme}.

We first consider the case for spatial semidiscretization \eqref{eq:mild-space}.
We define the associated Markov semigroup
$P^N_t \in \mathcal{L}(\mathcal B_b(\mathbb{H}_N^1) )$ as
\begin{equation*}
P^N_t \varphi ( X_0^N )=  \mathbb{E}
 \big[ \varphi ( X^N(t; X_0^N) ) \big],  \,
 \varphi \in \mathcal B_b(\mathbb{H}_N^1), \, t \ge 0,
\end{equation*}
where $\mathcal B_b(\mathbb{H}_N^1)$ is the set of Borel bounded measurable functions $\varphi: \mathbb{H}_N^1 \to \mathbb{R}$.
The dual operator of $\{P^N_t\}_{t\ge0}$ is given by
\begin{equation*}
\nu P_t^{N} (A) = \int_{\mathbb{H}_N^1} P_t^N \mathds{1}_A(U) \nu(\mathrm{d} U), \, A \in \mathcal{B}(\mathbb{H}^1_N), \,
\nu \in \mathcal{P}(\mathbb{H}_N^1).
\end{equation*}

By a similar argument as in the case of the exact solution, it can be shown that
the spatial semi-discretization \eqref{eq:mild-space} is exponentially ergodic with a unique invariant measure $\mu^N$. The proof is thus omitted here.

\begin{theorem}
Let \cref{ass:nonlinearity,ass:noise,ass:nonlinearity2} hold.
Then,  the spectral Galerkin semi-discretization \eqref{eq:mild-space} admits a unique invariant measure $\mu^N$ in $\mathbb{H}_N^1$ such that for any $t \ge 0, U \in \mathbb{H}_N^1$ and $\varphi\in \mathcal C_{p,\gamma}$, for $\epsilon > 0$ sufficiently small, there exists a constant $C(\epsilon,p,\gamma)>0$ such that
\begin{equation}\label{eq:exponential-ergodicity-SG}
\left| P_t^N \varphi (U)
     -  \int_{\mathbb H^1} \varphi \dd \mu^N \right|
     \leq C(\epsilon,p,\gamma) \, \|\varphi\|_{p,\gamma} \, ( 1 + \|U\|_{\mathbb{H}^1}^{\gamma+\frac p2}  )
     e^{- \epsilon \gamma t}.
\end{equation}
\end{theorem}

For the full discretization \eqref{eq:full-scheme}, we can also obtain the existence of an invariant measure by following the Krylov--Bogoliubov procedure as in the proof of \cref{prop:existence-IM-exact}.
The proof is omitted.
\begin{proposition}
Let \cref{ass:nonlinearity,ass:noise} hold.
Then there exists an invariant measure $\mu_\tau^N$ in $\mathbb{H}_N^1$ of the full discretization \eqref{eq:full-scheme}.
\end{proposition}

Let $\{P_n^{\tau,N}\}_{n \in \N}$ be the discrete Markov transition semigroup associated to  $\{ X_n^N \}_{n \in \N}$.
The next step is to show the uniqueness of the invariant measure $\mu_\tau^N$ for $\{P_n^{\tau,N}\}_{n \in \N}$,
which relies on the following lemma.

\begin{lemma} \label{contraction:full-discretization}
Assume that \cref{ass:nonlinearity,ass:noise,ass:nonlinearity2} hold, and $X_0,\widetilde{X}_0 \in \mathbb{H}^1$.
Let $X_n^N$ and $\widetilde{X}_n^N$ be the solutions of \eqref{eq:full-scheme} with different initial data ${\bf\Pi}_N X_0$ and ${\bf\Pi}_N \widetilde{X}_0$, respectively.
Then for $\epsilon > 0$ sufficiently small, $p>0$, and any $\tau \in (0,1)$,
\begin{equation}\label{eq:difference-discrete-pathwise}
 \| X_n^N - \widetilde{X}_n^N \|_{\mathbb{H}^1}
   \leq \tfrac{1+2\epsilon}{1-2\epsilon} \, e^{-\epsilon t_n}
       \| X_0 - \widetilde{X}_0 \|_{\mathbb{H}^1}^2, \,\text{ a.s. }\, n \in \N
\end{equation}
and
\begin{equation}\label{eq:difference-discrete}
\| X_n^N - \widetilde{X}_n^N \|_{L^p(\Omega;\mathbb{H}^1)}
   \leq \tfrac{1+2\epsilon}{1-2\epsilon} \, e^{-\epsilon t_n}
       \| X_0 - \widetilde{X}_0 \|_{L^p(\Omega;\mathbb{H}^1)} , \,\, n \in \N.
\end{equation}
Moreover, for any $\nu_1,\nu_2\in\mathcal P_p(\mathbb H^1_N)$
\begin{align}\label{eq:exponential-convergence-scheme}
\mathcal W_p(\nu_1 P_n^{\tau,N},\nu_2P_n^{\tau,N})\le
\tfrac{1+2\epsilon}{1-2\epsilon} \, e^{- \epsilon t_n} \mathcal W_p(\nu_1,\nu_2).
\end{align}
\end{lemma}

\begin{proof}
Denote $\bar{u}_{n+1}^N = u_{n+1}^N - \tilde{u}_{n+1}^N$ and
$\bar{v}_{n+1}^N = v_{n+1}^N - \tilde{v}_{n+1}^N$,
it follows that
\begin{align*}
\begin{cases}
\bar{u}_{n+1}^N= \bar{u}_n^N + \tau \, \bar{v}_{n+1}^N, \\
\bar{v}_{n+1}^N= \bar{v}_n^N - \tau
  \big( \Lambda_N \bar{u}_{n+1}^N + \eta\bar{v}_{n+1}^N
  +  \Pi_N F (v_{n+1}^N) -  \Pi_N F(\tilde{v}_{n+1}^N) \big).
\end{cases}
\end{align*}
On the one hand, we have
\begin{equation*}
\begin{split}
\langle  \bar{u}_{n+1}^N &  - \bar{u}_n^N,
     \bar{u}_{n+1}^N \rangle_{\dot{H}^1}
 + \langle \bar{v}_{n+1}^N  - \bar{v}_n^N,
     \bar{v}_{n+1}^N \rangle
  \\ & = \tfrac 12 \big( \| \bar{u}_{n+1}^N \|_{\dot{H}^1}^2
        - \| \bar{u}_{n}^N \|_{\dot{H}^1}^2
        + \| \bar{u}_{n+1}^N - \bar{u}_n^N \|_{\dot{H}^1}^2 \big)
  \\ & \quad + \tfrac 12 \big( \| \bar{v}_{n+1}^N \|^2
        - \| \bar{v}_{n}^N \|^2
        + \| \bar{v}_{n+1}^N - \bar{v}_n^N \|^2 \big)
  \\ & \geq \tfrac 12 \big( \| \bar{u}_{n+1}^N \|_{\dot{H}^1}^2
        - \| \bar{u}_{n}^N \|_{\dot{H}^1}^2
        + \| \bar{v}_{n+1}^N \|^2 - \| \bar{v}_{n}^N \|^2 \big).
\end{split}
\end{equation*}
On the other hand, we make use of \eqref{oneside-lip} to obtain
\begin{equation}\label{eq:full-u-v}
\begin{split}
\langle & \bar{u}_{n+1}^N  - \bar{u}_n^N,
          \bar{u}_{n+1}^N \rangle_{\dot{H}^1}
 + \langle \bar{v}_{n+1}^N - \bar{v}_n^N, \bar{v}_{n+1}^N \rangle
  \\ & = \tau \langle \bar{v}_{n+1}^N,
          \bar{u}_{n+1}^N \rangle_{\dot{H}^1}
         - \tau \langle \Lambda_N \bar{u}_{n+1}^N +\eta \bar{v}_{n+1}^N
         +  \Pi_N F(v_{n+1}^N) - \Pi_N F(\tilde{v}_{n+1}^N),
            \bar{v}_{n+1}^N \rangle
  \\ & = - \eta  \tau \| \bar{v}_{n+1}^N \|^2
         - \tau \langle F(v_{n+1}^N) -  F(\tilde{v}_{n+1}^N), \bar{v}_{n+1}^N \rangle
  \\ & \leq - ( \eta + a_2 ) \tau \| \bar{v}_{n+1}^N \|^2.
\end{split}
\end{equation}
Next, using the Cauchy--Schwartz inequality and  \cref{ass:nonlinearity2}, we  deduce
\begin{equation*}
\begin{split}
 \epsilon & \big(  \langle \bar{u}_{n+1}^N  - \bar{u}_n^N,
              \bar{v}_{n+1}^N \rangle
              + \langle \bar{u}_{n+1}^N,
              \bar{v}_{n+1}^N  - \bar{v}_n^N \rangle \big)
   \\ & =  \epsilon \big(  \tau \| \bar{v}_{n+1}^N \|^2
         - \tau \langle \bar{u}_{n+1}^N,
            \Lambda_N  \bar{u}_{n+1}^N + \eta \bar{v}_{n+1}^N
            +  \Pi_N F (v_{n+1}^N) -  \Pi_N F (\tilde{v}_{n+1}^N)
               \rangle \big)
   \\ & = \epsilon \big(  \tau \| \bar{v}_{n+1}^N \|^2
              - \tau \| \bar{u}_{n+1}^N \|_{\dot{H}^1}^2
              - \tau \eta \langle \bar{u}_{n+1}^N, \bar{v}_{n+1}^N \rangle
              - \tau \langle \bar{u}_{n+1}^N,
                F (v_{n+1}^N) -   F (\tilde{v}_{n+1}^N) \rangle \big)
   \\ & \leq \epsilon \big(  \tau \| \bar{v}_{n+1}^N \|^2
              - \tau \| \bar{u}_{n+1}^N \|_{\dot{H}^1}^2
              - \tau \eta  \langle \bar{u}_{n+1}^N, \bar{v}_{n+1}^N \rangle
              + L \, \tau \| \bar{u}_{n+1}^N \|_{\dot{H}^1}              \| \bar{v}_{n+1}^N \| \big).
\end{split}
\end{equation*}
Following the similar approach in \cref{lem:exp-convergence},
we obtain, for sufficiently small $\epsilon > 0$,
\begin{align*}
\langle & \bar{u}_{n+1}^N   - \bar{u}_n^N,
              \bar{u}_{n+1}^N \rangle_{\dot{H}^1}
  + \langle \bar{v}_{n+1}^N  - \bar{v}_n^N,
              \bar{v}_{n+1}^N \rangle
 \\ & \quad + 8 \epsilon \big( \langle \bar{u}_{n+1}^N  - \bar{u}_n^N,
              \bar{v}_{n+1}^N \rangle
              + \langle \bar{u}_{n+1}^N,
              \bar{v}_{n+1}^N  - \bar{v}_n^N \rangle \big)
 \\ & \leq - \tau \big[ ( \eta  + a_2 ) - 8 \epsilon \big]
               \| \bar{v}_{n+1}^N \|^2
           - 8 ~\tau \, \epsilon \| \bar{u}_{n+1}^N \|_{\dot{H}^1}^2
 \\ & \quad  - 8 ~ \tau \,\eta \, \epsilon
                \langle \bar{u}_{n+1}^N, \bar{v}_{n+1}^N \rangle
             + 8 ~\tau \,L \, \epsilon \| \bar{u}_{n+1}^N \|_{\dot{H}^1}              \| \bar{v}_{n+1}^N \|
 \\ & \leq - \tau \big( \tfrac {\eta +a_2}{2} \| \bar{v}_{n+1}^N \|^2
             + 4 \epsilon \| \bar{u}_{n+1}^N \|_{\dot{H}^1}^2 \big)
 \\ & \leq - 2 \epsilon \, \tau
        \big( \| \bar{u}_{n+1}^N \|_{\dot{H}^1}^2
        + \| \bar{v}_{n+1}^N \|^2
        + 8 \epsilon \langle \bar{u}_{n+1}^N, \bar{v}_{n+1}^N \rangle \big).
 \\ & \leq - 2 \epsilon \, \tau
        \big( \tfrac 12 \| \bar{u}_{n+1}^N \|_{\dot{H}^1}^2
        + \tfrac 12 \| \bar{v}_{n+1}^N \|^2
        + 8 \epsilon \langle \bar{u}_{n+1}^N, \bar{v}_{n+1}^N \rangle \big).
\end{align*}
Applying elementary inequality $ab \geq - \tfrac12 (a^2 + b^2)$ yields
\begin{align*}
\langle & \bar{u}_{n+1}^N   - \bar{u}_n^N,
              \bar{u}_{n+1}^N \rangle_{\dot{H}^1}
  + \langle \bar{v}_{n+1}^N  - \bar{v}_n^N,
              \bar{v}_{n+1}^N \rangle
 \\ & \quad + 8 \epsilon \big( \langle \bar{u}_{n+1}^N  - \bar{u}_n^N,
              \bar{v}_{n+1}^N \rangle
              + \langle \bar{u}_{n+1}^N,
              \bar{v}_{n+1}^N  - \bar{v}_n^N \rangle \big)
 \\ & = \tfrac 12 \| \bar{u}_{n+1}^N \|_{\dot{H}^1}^2
       - \tfrac 12 \| \bar{u}_{n}^N \|_{\dot{H}^1}^2
       + \tfrac 12 \| \bar{u}_{n+1}^N - \bar{u}_n^N \|_{\dot{H}^1}^2
 \\ & \quad + \tfrac 12 \| \bar{v}_{n+1}^N \|^2
       - \tfrac 12 \| \bar{v}_{n}^N \|^2
       + \tfrac 12 \| \bar{v}_{n+1}^N - \bar{v}_{n}^N \|^2
 \\ & \quad + 8 \epsilon \big(
              \langle \bar{u}_{n+1}^N , \bar{v}_{n+1}^N \rangle
              - \langle \bar{u}_{n}^N, \bar{v}_{n}^N \rangle
              + \langle \bar{u}_{n+1}^N  - \bar{u}_n^N,
              \bar{v}_{n+1}^N  - \bar{v}_n^N \rangle \big)
 \\ & \geq \tfrac 12 \| \bar{u}_{n+1}^N \|_{\dot{H}^1}^2
           - \tfrac 12 \| \bar{u}_{n}^N \|_{\dot{H}^1}^2
           + \tfrac 12 \| \bar{v}_{n+1}^N \|^2
           - \tfrac 12 \| \bar{v}_{n}^N \|^2
 \\ & \quad + 8 \epsilon \big(
              \langle \bar{u}_{n+1}^N , \bar{v}_{n+1}^N \rangle
              - \langle \bar{u}_{n}^N, \bar{v}_{n}^N \rangle \big).
\end{align*}
Hence, using the inequality $a^k < e^{-(1-a)k},a \in (0,1)$ gives
\begin{equation*}
 \tfrac 12 \|  \bar{u}_n^N \|_{\dot{H}^1}^2
      + \tfrac 12 \| \bar{v}_n^N \|^2  + 8 \epsilon \,
       \langle \bar{u}_n^N,  \bar{v}_n^N \rangle
 \leq e^{- \epsilon t_n}
         \left( \tfrac 12 \| \bar{u}_0^N \|_{\dot{H}^1}^2
        +  \tfrac 12 \| \bar{v}_0^N \|^2  + 8 \epsilon \,
         \langle \bar{u}_0^N, \bar{v}_0^N \rangle \right),
\end{equation*}
which yields \eqref{eq:difference-discrete-pathwise}.
And we conclude \eqref{eq:difference-discrete} by taking $p$th moment on the above estimate. The proof of \eqref{eq:exponential-convergence-scheme} is similar to that of \eqref{eq:exponential-convergence}, so it is omitted.
\end{proof}

Now we can show that  \eqref{eq:full-scheme} is ergodic with a unique invariant measure $\mu_\tau^N$, which converges exponentially to the equilibrium.
The proof is similar to that of \cref{thm:exp-ergodic} and is omitted.

\begin{theorem}\label{thm:ergodicity-space}
Let \cref{ass:nonlinearity,ass:noise,ass:nonlinearity2} hold.
Then,  the full discretization \eqref{eq:full-scheme} admits a unique invariant measure $\mu^N_{\tau}$ in $\mathbb{H}_N^1$ such that for any $n \in \N, U \in \mathbb{H}_N^1$, sufficiently small $\epsilon>0$, and continuous mapping $\varphi\in \mathcal C_{p,\gamma}( \mathbb{H}_N^1,\R)$,
\begin{equation}\label{eq:exp-full-dis}
\left| P_n^{\tau,N} \varphi (U)
     -  \int \varphi \dd \mu^N_{\tau} \right|
     \leq C(\epsilon,p,\gamma) \|\varphi\|_{p,\gamma} \,
        \, ( 1 + \|U\|_{\mathbb{H}^1}^{\gamma+\frac p2}  )
        e^{- \epsilon \gamma t_n}.
\end{equation}
\end{theorem}

\section{Approximation for invariant measure}
\label{Sec:approximation}
This section presents estimates on the approximation of the invariant measure. More precisely, we obtain the error estimates of invariant measures both in Wasserstein distance and in the weak sense by deriving the strong error estimate of the full discretization. The errors between the time average of the exact and numerical solutions and the ergodic limit are estimated, which leads to the strong laws of large numbers of the exact solution and the full discretization.

\subsection{Strong error analysis for full discretization}\label{Sec:error}

In  this subsection, we concentrate on the error analysis of the full discretization \eqref{eq:full-scheme}.
Due to that the operator semigroup $E(t)$ of  \eqref{SWE} lacks of smoothing effect, it is difficult to derive the  estimates in the maximum norm of solutions directly.
Instead, the Sobolev embedding theorem has to be used to bound the maximum norm by using the higher regularity of the solution.
Hence, we focus on the case $d=1$, since the embedding $\dot{H}^1 \hookrightarrow C(\mathcal{O};\R)$ fails for $d \ge 2$.

\subsubsection{Spatial error analysis}

In this part, we analyze the error of spectral Galerkin method \eqref{eq:spectral-galerkin} applied to \eqref{SWE}.
To start the convergence analysis, we need an additional assumption on $F$.
\begin{assumption}\label{ass:nonlinearity3}
The continuously differentiable function $f: \mathbb{R} \to \mathbb{R}$
associated to  $F$
satisfies
\begin{align*}
\left| f' (\xi) \right| & \leq C (1+|\xi|), \, \xi \in \R.
\end{align*}
\end{assumption}

The following result concerns the spatial error analysis.
\begin{proposition}\label{prop:spatial-error}
Under \cref{ass:nonlinearity,ass:noise,ass:nonlinearity3} and
assume that $X_0 \in \mathbb{H}^2$.
Let $X(t)$ and $X^N(t)$ be the solutions of \eqref{SWE1} and \eqref{eq:mild-space}, respectively.
Then, for any $t \ge 0$, $ p \ge 1$ and $N \in \N$,
there exists a positive constant $C$, independent of $t$, such that
\begin{equation*}
\big\| X(t) - X^N(t) \big\|_{L^p(\Omega;\mathbb{H}^1)}
\leq C \lambda_N^{-\frac12} (1+t^{\frac12}).
\end{equation*}
\end{proposition}

\begin{proof}
We first decompose $\| X(t) - X^N(t) \|_{L^p(\Omega;\mathbb{H}^1)}$ into two parts as follows,
\begin{equation*}
 \| X(t)  - X^N(t) \|_{L^p(\Omega;\mathbb{H}^1)}
   \leq \| X(t) - \mathbf\Pi_N X(t) \|_{L^p(\Omega;\mathbb{H}^1)}
         +  \|  \mathbf\Pi_N X(t) - X^N(t) \|_{L^p(\Omega;\mathbb{H}^1)},
\end{equation*}
where
\begin{equation*}
\mathbf\Pi_N X(t) = E_N (t) X_0^N
   - \int_0^t E_N ( t - s ) \mathbf{F}_N ( X(s) ) \dd s
   + \int_0^t E_N ( t - s ) B_N \dd W(s).
\end{equation*}
For the estimate of the term $\| X(t) - \mathbf\Pi_N X(t) \|_{L^p(\Omega;\mathbb{H}^1)}$, applying \eqref{eq:P_N-estimate} and \cref{prop:bound-moment}  gives
\begin{equation*}
 \E \big[ \| X(t) - \mathbf\Pi_N X(t) \|_{\mathbb{H}^1}^p \big]
    \leq C \, \lambda_N^{-\frac p2} \,
    \E [ \| X(t) \|_{\mathbb{H}^2}^p ]
    \leq C \, \lambda_N^{-\frac p2}.
\end{equation*}
Next we estimate the term $\|  \mathbf\Pi_N X(t) - X^N(t) \|_{L^p(\Omega;\mathbb{H}^1)}$. Denoting $e^N(t) =  \mathbf\Pi_N X(t) - X^N(t)
=:\big( e^N_1(t),e^N_2(t) \big)^{\top}$, we have
\begin{equation*}
\frac{\dd e^N(t)}{\dd t} = A_N e^N(t)
   + \mathbf{F}_N ( X^N(t) )  - \mathbf{F}_N ( X(t) ).
\end{equation*}
 One obtains
\begin{equation*}
\begin{split}
\frac{\dd \| e^N(t) \|_{\mathbb{H}^1}^2}{\dd t} & = 2 \left\langle e^N(t), \frac{\dd e^N(t)}{\dd t}  \right \rangle_{\mathbb{H}^1}
  \\ & = 2 \left\langle e^N(t), A_N e^N(t)
   + \mathbf{F}_N ( X^N(t) )  - \mathbf{F}_N ( X(t) )      \right \rangle_{\mathbb{H}^1}
  \\ & = 2   \left\langle e^N(t),
        \mathbf{F}_N ( X^N(t) )  - \mathbf{F}_N ( X(t) )  \right \rangle_{\mathbb{H}^1},
\end{split}
\end{equation*}
where the final step utilized the antisymmetry of
$A_N$. Subsequently, through the application of the Cauchy--Schwartz inequality and \eqref{oneside-lip}, one obtain
\begin{equation*}
\begin{split}
\frac{\dd \| e^N(t) \|_{\mathbb{H}^1}^2}{\dd t}
     & = 2 \left\langle e^N(t),
        \mathbf{F}_N ( X^N(t) )  - \mathbf{F}_N (  {\bf\Pi}_N X(t) )
          \right \rangle_{\mathbb{H}^1}
  \\ & \quad + 2 \left\langle e^N(t),
        \mathbf{F}_N (  {\bf \Pi_N} X(t) )  - \mathbf{F}_N ( X(t) )
          \right \rangle_{\mathbb{H}^1}
  \\ & \leq - ( \eta + a_2 ) \| e^N_2 (t) \|^2
     + C \big\| \mathbf{F}_N (  {\bf\Pi_N} X(t) )
        - \mathbf{F}_N ( X(t) )  \big\|_{\mathbb{H}^1}^2.
\end{split}
\end{equation*}
Integrating both sides from 0 to $t$  gives
\begin{equation*}
\| e^N(t) \|_{\mathbb{H}^1}^2
 \leq C \int_0^t ( 1 + \| v(s) \|_{\dot{H}^1}^2 )
         \|  \Pi_N v(s) - v (s) \|^2 \dd s,
\end{equation*}
where we used \cref{ass:nonlinearity3} and the Sobolev embedding $\dot{H}^1 \hookrightarrow C(\mathcal{O};\R)$ for $d=1$.
In addition, taking expectations on both sides and using H\"{o}lder's inequality yield
\begin{equation*}
\begin{split}
 \| e^N(t) \|_{L^{2p}(\Omega;\mathbb{H}^1)}^2
  & \leq  C \int_0^t  ( 1 + \| v(s) \|_{L^{4p}(\Omega;\dot{H}^1)}^2 )
    \|  \Pi_N v(s) - v (s) \|_{L^{4p}(\Omega;H)}^2 \dd s
  \\ & \leq  C \lambda_N^{- 1}
  \int_0^t  \Big( 1 +  \| v(s) \|_{L^{4p}(\Omega;\dot{H}^1)}^4  \Big) \dd s
   \leq  C \lambda_N^{- 1} t.
\end{split}
\end{equation*}
The proof is thus completed.
\end{proof}

\subsubsection{Temporal error analysis}
The full discretization \eqref{eq:full-scheme} can be rewritten as
\begin{equation}\label{eq:full-scheme1}
X_k^N = E_{N,\tau}^k \, X_0^N
    - \tau \sum_{i=0}^{k-1} E_{N,\tau}^{k-i} \, \mathbf{F}_N (X_{i+1}^N)
    + \sum_{i=0}^{k-1} E_{N,\tau}^{k-i} \, B_N \delta W_i,
\end{equation}
where we denote $E_{N,\tau}:= (I - \tau A_N)^{-1}$.
The following lemma presents estimates of semigroups, whose proofs can be found in
 \cite[Equation (13)]{Anton2016full} and \cite[Theorem 4]{MR537280}.

\begin{lemma}\label{lem:error-estimate}
For $t \ge s \ge 0$, we have
\begin{equation}\label{eq:error-determ-continuous}
\| ( E(t) - E(s) ) X \|_{\mathbb{H}}
    \leq C ( t - s)^{\gamma} \| X \|_{\mathbb{H}^{\gamma}},
    \,  X \in \mathbb{H}^{\gamma}, \, \gamma \in [0,1].
\end{equation}
Moreover, for $n \in \N$,
\begin{equation}\label{eq:error-determ-discrete}
\| ( E_N(t_n) - E_{N,\tau}^n ) X \|_{\mathbb{H}^1}
    \leq C \, t_n^{\frac\gamma2}\tau^{\frac\gamma2} \| X \|_{\mathbb{H}^{\gamma+1}},
    \,  X \in \mathbb{H}^{\gamma}, \, \gamma =0,1,2.
\end{equation}
\end{lemma}

\begin{lemma}\label{lem:XN-regularity}
Let \cref{ass:nonlinearity,ass:noise} hold. Assume that
$X_0 \in \mathbb{H}^2$. Then for $0 \leq s \leq t$ and $p \ge 1$, it holds 
  that,
\begin{equation}\label{eq:time-holder-continuity}
 \| X^N(t) - X^N(s) \|_{L^p(\Omega;\mathbb{H}^1)}
  \leq C | t - s |^{\frac12}.
\end{equation}
\end{lemma}
\begin{proof}
Based on the mild solution of \eqref{eq:mild-space}, we obtain
\begin{equation*}
\begin{split}
X^N(t) - X^N(s) &= \big( E_N(t-s) - I \big) X^N(s)
     \\& \quad - \int_s^t E_N(t-r) {\bf F}_N (X^N(r)) \,\dd r
     + \int_s^t E_N(t-r) B_N \,\dd W(r),
\end{split}
\end{equation*}
which yields
\begin{equation}\label{eq.holdersemi}
\begin{split}
\| X^N(t) -  X^N(s) \|_{L^p(\Omega; \mathbb{H}^1)}
  & \leq \| ( E_N(t-s) - I ) X^N(s) \|_{L^p(\Omega; \mathbb{H}^1)}
  \\& \quad + \Big\| \int_s^t E_N(t-r) {\bf F}_N(X^N(r)) \,\dd r
                          \Big\|_{L^p(\Omega; \mathbb{H}^1)}
  \\& \quad + \Big\| \int_s^t E_N(t-r) B_N \,\dd W(r)
                          \Big\|_{L^p(\Omega; \mathbb{H}^1)}.
\end{split}
\end{equation}
We estimate the above three terms separately.
The first term in the right side of \eqref{eq.holdersemi} can be directly estimated by \eqref{eq:error-determ-continuous},
\begin{equation*}
\big\| ( E_N(t-s) - I ) X^N(s) \big\|_{L^p(\Omega; \mathbb{H}^{1})}
\leq C | t - s |
          \| X^N(s) \|_{L^p(\Omega; \mathbb{H}^2)}
\leq C  | t - s |.
\end{equation*}
For the second term in the right side of \eqref{eq.holdersemi}, it follows from the stability of the semigroup that
\begin{equation*}
\begin{split}
\Big\| \int_s^t  E_N(t-r) & {\bf F}_N (X^N(r))
\,\dd r   \Big\|_{L^p(\Omega; \mathbb{H}^{1})}
  \leq   \int_s^t
 \big\| {\bf F}_N (X^N(r)) \big\|_{L^p(\Omega; \mathbb{H}^{1})} \,\dd r
\\ & \leq C \int_s^t \big( 1 +
 \big\|  X^N(r) \big\|_{L^p(\Omega; \mathbb{H}^{1})} \big) \,\dd r
 \leq C  | t - s |.
\end{split}
\end{equation*}
Finally, using the Burkholder--Davis--Gundy-type inequality, the stability of the sine and cosine operators as well as of the projection operator,
we obtain
\begin{equation*}
\begin{split}
\Big\|  \int_s^t & E_N(t-r) B_N \dd W(r)
           \Big\|_{L^p(\Omega; \mathbb{H}^{1})}
\\ & \leq \left( \int_s^t
    \big\| \sin \big( (t-r) \Lambda^{\frac12} \big)
      \big\|_{\mathcal{L}_2^0}^2
  + \big\| \cos \big( (t-r) \Lambda^{\frac12} \big)
      \big\|_{\mathcal{L}_2^0}^2 \,\dd r  \right)^{\frac12}
\\& \leq C | t - s |^{\frac12}.
\end{split}
\end{equation*}
The proof is thus completed.
\end{proof}

\begin{theorem} \label{thm:time-error}
Let \cref{ass:nonlinearity,ass:noise,ass:nonlinearity3} hold, 
$X_0 \in \mathbb{H}^2$, and $X^N(t)$ and $X^N_k$ be the solutions of  \eqref{eq:mild-space} and \eqref{eq:full-scheme1}, respectively.
Then, for any $p \ge 2$, it holds that
\begin{equation}\label{eq:mian-result}
\| X^N(t_k) - X_k^N \|_{L^p(\Omega;\mathbb{H}^1)}
 \leq C \, \tau^{\frac12} ( 1 + t_k^{3} ),
 \quad k \in \N.
\end{equation}
\end{theorem}

\begin{proof}
We introduce the auxiliary process
\begin{equation*}
\widehat{X}_k^N = E_{N,\tau}^k \, X_0^N
    - \tau \sum_{i=0}^{k-1} E_{N,\tau}^{k-i} \, \mathbf{F}_N (X^N(t_{i+1}))
    + \sum_{i=0}^{k-1} E_{N,\tau}^{k-i} \, B_N \delta W_i.
\end{equation*}
By using the stability of $E_{N,\tau}^{k}$,
it can be derived that
\begin{equation*}
 \| \widehat{X}_k^N \|_{L^p(\Omega;{\mathbb H}^2)} 
\leq C ( 1 + t_k).
\end{equation*}
Then, we divide
$\| X^N(t_k) - X_k^N \|_{L^p(\Omega;\mathbb{H}^1)}$ into two terms,
\begin{equation*}
 \| X^N(t_k) - X_k^N \|_{L^p(\Omega;\mathbb{H}^1)}
  \leq  \| X^N(t_k) - \widehat{X}_k^N \|_{L^p(\Omega;\mathbb{H}^1)}
     +  \| \widehat{X}_k^N - X_k^N \|_{L^p(\Omega;\mathbb{H}^1)}.
\end{equation*}
For the first term $\| X^N(t_k) - \widehat{X}_k^N \|_{L^p(\Omega;\mathbb{H}^1)}$, we further decompose it as follows,
\begin{equation*}
\begin{split}
\| & X^N(t_k) - \widehat{X}_k^N \|_{L^{4p}(\Omega;\mathbb{H}^1)}
   \\ & \leq  \big\| ( E_N(t_k) - E_{N,\tau}^k ) X_0^N
 \big\|_{L^{4p}(\Omega;\mathbb{H}^1)}
   \\ & \quad + \Big\| \int_0^{t_k}
     E_N (t_k-s) \mathbf{F}_N (X^N(s)) \dd s
         - \tau \sum_{i=0}^{k-1} E_{N,\tau}^{k-i} \, \mathbf{F}_N (X^N(t_{i+1})) \Big\|_{L^{4p}(\Omega;\mathbb{H}^1)}
   \\ & \quad + \Big\| \int_0^{t_k} E_N (t_k-s) B_N \dd W(s)
         - \sum_{i=0}^{k-1} E_{N,\tau}^{k-i} \, B_N \delta W_i \Big\|_{L^{4p}(\Omega;\mathbb{H}^1)}
   \\ & =: I_1 + I_2 + I_3.
\end{split}
\end{equation*}
Using \eqref{eq:error-determ-discrete} in \cref{lem:error-estimate}
shows
\begin{equation*}
I_1 \leq  C ~ \tau^{\frac12} \| X_0 \|_{\mathbb{H}^2}
 t_k^{\frac12}.
\end{equation*}
For $I_2$, we obtain
\begin{equation*}
\begin{split}
I_2  &  \leq \left\| \sum_{i=0}^{k-1} \int_{t_i}^{t_{i+1}}
               E_N(t_k-s) \big( \mathbf{F}_N(X^N(s))
          - \mathbf{F}_N(X^N(t_{i+1})) \big) \dd s
          \right\|_{L^{4p}(\Omega;\mathbb{H}^1)}
   \\ & \quad  + \left\| \sum_{i=0}^{k-1} \int_{t_i}^{t_{i+1}}
         \big(  E_N(t_k-s) - E_{N,\tau}^{k-i} \big)
           \mathbf{F}_N(X^N(t_{i+1})) \dd s
           \right\|_{L^{4p}(\Omega;\mathbb{H}^1)}
   \\ & =: I_{2,1} + I_{2,2}.
\end{split}
\end{equation*}
By using the stability of $E_N(t)$, Taylor's expansion and  \cref{lem:XN-regularity},
one derives
\begin{align*}
I_{2,1} & \leq \sum_{i=0}^{k-1} \int_{t_i}^{t_{i+1}} \left\|
  \mathbf{F}_N(X^N(s)) - \mathbf{F}_N(X^N(t_{i+1}))
  \right\|_{L^{4p}(\Omega;\mathbb{H}^1)}  \dd s
    \\ & \leq \sum_{i=0}^{k-1} \int_{t_i}^{t_{i+1}}
    \big( 1 + \| v^N(s) \|_{L^{8p}(\Omega;\dot{H}^1)}
    + \| v^N(t_{i+1}) \|_{L^{8p}(\Omega;\dot{H}^1)} \big)
    \\ & \qquad \qquad \times  \big\| X^N(s) - X^N(t_{i+1}) \big\|_{L^{8p}(\Omega;\mathbb{H}^1)} \dd s
    \\ & \leq C \, \tau^{\frac12} t_k.
\end{align*}
For the term $I_{2,2}$,
utilizing \eqref{eq:error-determ-continuous} and
\eqref{eq:error-determ-discrete} shows that
\begin{align*}
I_{2,2} & \leq \left\| \sum_{i=0}^{k-1} \int_{t_i}^{t_{i+1}}
         \big(  E_N( t_k - s ) - E_N( t_k - t_i ) \big)
           \mathbf{F}_N(X^N(t_{i+1})) \dd s
           \right\|_{L^{4p}(\Omega;\mathbb{H}^1)}
    \\ & \quad + \left\| \sum_{i=0}^{k-1} \int_{t_i}^{t_{i+1}}
         \big(  E_N(t_{k-i}) - E_{N,\tau}^{k-i} \big)
           \mathbf{F}_N(X^N(t_{i+1})) \dd s
           \right\|_{L^{4p}(\Omega;\mathbb{H}^1)}
    \\ & \leq C \sum_{i=0}^{k-1} \int_{t_i}^{t_{i+1}}
              \tau \, ( 1 + \| X^N (t_{i+1}) \|_{L^{8p}(\Omega;\mathbb{H}^2)}^2) \dd s
    \\ & \quad  + C \sum_{i=0}^{k-1} \int_{t_i}^{t_{i+1}}
      \tau^{\frac12} \cdot t_k^{\frac12}  \cdot
     ( 1 + \| X^N (t_{i+1}) \|_{L^{8p}(\Omega;\mathbb{H}^2)}^2 ) \dd s
  \\ & \leq C \, \tau^{\frac12} t_k^{\frac32}.
\end{align*}
For the estimate of $I_3$,
\begin{align*}
I_3  &  \leq  \left\| \sum_{i=0}^{k-1} \int_{t_i}^{t_{i+1}}
              \big( E_N (t_k - s) - E_N (t_k - t_i) \big)
                B_N  \dd W(s) \right\|_{L^{4p}(\Omega;\mathbb{H}^1)}
   \\ & \quad  + \left\| \sum_{i=0}^{k-1} \int_{t_i}^{t_{i+1}}
         \big(  E_N(t_{k-i}) - E_{N,\tau}^{k-i} \big)
                B_N  \dd W(s) \right\|_{L^{4p}(\Omega;\mathbb{H}^1)}
   \\ & =: I_{3,1} + I_{3,2}.
\end{align*}
Using Burkholder--Davis--Gundy-type inequality, \eqref{eq:error-determ-continuous} and \eqref{eq:A-Q}, we arrive at
\begin{equation*}
\begin{split}
I_{3,1} & \leq \left( \sum_{i=0}^{k-1} \int_{t_i}^{t_{i+1}}
        \big\| \big( E_N(t_k - s) - E_N(t_k - t_i) \big)
        B_N \big\|_{\mathcal{L}_2^0(\mathbb{H}^1)}^2 \dd s
        \right)^{\frac12}
   \\ & \leq C \left( \tau^2 \int_0^{t_k} \big\|
       \Lambda^{\frac12} Q^{\frac12}
       \big\|_{\mathcal{L}_2}^2 \dd s \right)^{\frac12}
 \leq C \, \tau \, t_k^{\frac12}.
\end{split}
\end{equation*}
Analogously, using \eqref{eq:error-determ-discrete} instead of \eqref{eq:error-determ-continuous}, we possess
\begin{align*}
I_{3,2} & \leq \left( \sum_{i=0}^{k-1} \int_{t_i}^{t_{i+1}}
               \big\| \big( E_N(t_{k-i})  - E_{N,\tau}^{k-i} \big)
               B_N \big\|_{\mathcal{L}_2^0(\mathbb{H}^1)}^2 \dd s \right)^{\frac12}
   \\ & \leq C \left( \tau \cdot t_k \cdot \int_0^{t_k}
            \big\| \Lambda^{\frac12} Q^{\frac12}
                   \big\|_{\mathcal{L}_2}^2 \dd s \right)^{\frac12}
    \leq C  \tau^{\frac 12}  t_k.
\end{align*}
Gathering the estimates of $I_1,I_2,$ and $I_3$, we have
\begin{equation}\label{eq:error-proof-1}
 \| X^N(t_k) - \widehat{X}_k^N \|_{L^{4p}(\Omega;\mathbb{H}^1)}
      \leq C \, \tau^{\frac12} t_k^\frac32.
\end{equation}
Next, we focus on the estimate of $\| \widehat{X}_k^N - X_k^N \|_{L^p(\Omega;\mathbb{H}^1)}$.
For convenience, we define
\begin{equation*}
e_k^N := \widehat{X}_k^N - X_k^N
       = \tau \sum_{i=0}^{k-1} E_{N,\tau}^{k-i}
        \big(\mathbf{F}_N (X^N_{i+1}) - \mathbf{F}_N (X^N(t_{i+1}))\big),
\end{equation*}
which satisfies
\begin{equation*}
e_k^N - e_{k-1}^N - \tau A_N e_k^N
    = \tau \big( \mathbf{F}_N (X^N_{k})
        - \mathbf{F}_N (X^N(t_{k})) \big),
        \quad e_0^N = 0.
\end{equation*}
Multiplying both sides by $e_k^N$, we have
\begin{equation*}
\langle e_k^N - e_{k-1}^N, e_k^N \rangle_{\mathbb{H}^1}
   = \tau \langle \mathbf{F} (X^N_{k}) - \mathbf{F} (\widehat{X}^N_{k}),
                 e_k^N \rangle_{\mathbb{H}^1}
   + \tau \langle \mathbf{F} (\widehat{X}^N_{k}) - \mathbf{F} (X^N(t_{k})),e_k^N \rangle_{\mathbb{H}^1}.
\end{equation*}
Using the inequality
$\langle a-b,a\rangle_{\mathbb{H}^1} \ge \frac12 (\|a\|_{\mathbb{H}^1}^2 - \|b\|_{\mathbb{H}^1}^2)$ and Young's inequality,
we obtain
\begin{equation*}
\tfrac12 \big( \| e_k^N \|_{\mathbb{H}^1}^2
          - \| e_{k-1}^N \|_{\mathbb{H}^1}^2 \big)
      \leq - \tfrac{ \eta + a_2 }2 ~ \tau ~ \| e_{k,2}^N \|^2
       + C \, \tau \, \| \mathbf{F} (\widehat{X}^N_{k}) - \mathbf{F} (X^N(t_{k})) \|_{\mathbb{H}^1}^2,
\end{equation*}
where $e_k^N:=(e_{k,1}^N,e_{k,2}^N)^{\top}$.
As a result, utilizing \cref{ass:nonlinearity3}, moment bounds of $\widehat{X}_k^N$ and $X^N(t)$ as well as \eqref{eq:error-proof-1} shows that
\begin{equation*}
\begin{split}
\| e_k^N \|_{L^{2p}(\Omega;\mathbb{H}^1)}^2
       & \leq   C \, \tau \, \sum_{i=1}^k
        \| \mathbf{F} (\widehat{X}^N_{i}) - \mathbf{F} (X^N(t_{i})) \|_{L^{2p}(\Omega;\mathbb{H}^1)}^2
    \\ & \leq   C \, \tau \, \sum_{i=1}^k
       ( 1 + \| v^N(t_i)\|_{L^{4p}(\Omega;\dot{H}^1)}^2
         + \|\widehat{v}_i^N \|_{L^{4p}
        (\Omega;\dot{H}^1)}^2 )
         \\ & \qquad \times  \| X^N(t_{i}) - \widehat{X}^N_{i}  \|_{L^{4p}(\Omega;\mathbb{H}^1)}^2
    \\ & \leq C \, \tau ( 1 + t_k^6 ).
\end{split}
\end{equation*}
Therefore, this together with \eqref{eq:error-proof-1} leads to
\begin{equation*}
 \| X^N(t) - X_k^N \|_{L^{2p}(\Omega;\mathbb{H}^1)}
 \leq C \, \tau^{\frac12} ( 1 + t_k^3 ),
\end{equation*}
which finishes the proof.
\end{proof}

\subsection{Error estimates of invariant measures}

This subsection presents error estimates of invariant measures both in Wasserstein distance and in the weak sense. 

Let $\mu$ and $\mu_\tau^N$ be the unique invariant measures of \eqref{SWE1} and \eqref{eq:full-scheme}, respectively.
The following theorem states the first result about the error estimate between $\mu_\tau^N$ and $\mu$ in  Wasserstein distance.

\begin{theorem}
Let conditions in \cref{prop:spatial-error}  and  \cref{ass:nonlinearity2} hold. 
Then for any $p\ge1$, there exists a constant $C(p,\epsilon)>0$ such that
\begin{align*}
\mathcal{W}_p\big(\mu, \mu_\tau^N \big)\le C(p,\epsilon)\big(\lambda_N^{-\frac 12}+  \tau^{\frac12}\big).
\end{align*}
\end{theorem}
\begin{proof}
For any $n\in\mathbb N$, one has
\begin{align*}
\mathcal{W}_p\big(\mu, \mu_\tau^N \big)=&\,\mathcal{W}_p\big(\mu P_{t_n}, \mu_\tau^N P_n^{\tau,N}\big)\\   \le&\,\mathcal{W}_p\big(\mu P_{t_n}, ({\bf \Pi}_N)_{\#}\mu P_{t_n}^{N}\big)+\mathcal{W}_p\big(({\bf \Pi}_N)_{\#}\mu P_{t_n}^N, ({\bf \Pi}_N)_{\#}\mu P_n^{\tau,N}\big)
\\ & \quad+\mathcal{W}_p\big(({\bf \Pi}_N)_{\#}\mu P_n^{\tau,N}, \mu_\tau^N P_n^{\tau,N} \big)
\end{align*}
with $({\bf \Pi}_N)_{\#}\mu$ being the push-forward of $\mu$ through ${\bf \Pi}_N$. For any $\epsilon>0$, it follows from \eqref{eq:exponential-convergence-scheme} that
\begin{align}\label{eq:third-term}
\mathcal{W}_p\big(({\bf \Pi}_N)_{\#}\mu P_n^{\tau,N}, \mu_\tau^N P_n^{\tau,N} \big)\le
\tfrac{1+2\epsilon}{1-2\epsilon} \, e^{- \epsilon t_n} \mathcal{W}_p\big(({\bf \Pi}_N)_{\#}\mu, \mu_\tau^N \big).
\end{align}
In order to use \cref{prop:spatial-error,thm:time-error} to estimate the remaining two terms, we need to show $\mu(\|\cdot\|_{\mathbb H^2}^p)<\infty$ for any $p\ge1$.
We claim that \eqref{eq:average-X} yields $\mu(\|\cdot\|_{\mathbb H^2}^p)<\infty$. 
Indeed, by tightness,  $\{\mu_T^0:=\frac1T\int_0^T\delta_0P_t\dd t\}_{T\ge 0}$ has a weakly convergent subsequence, saying $\{\mu_{T_n}^0\}_{n\ge 0}$, which weakly converges to the invariant measure $\mu$. 
For any $R>0$, the functional $\pi\mapsto\int_{\mathbb H^1}(\|x\|_{\mathbb H^2}^p\wedge R)\pi(\dd x)$ from $\mathcal P(\mathbb H^1)$ to $\R$ is lower semicontinuous w.r.t. weak convergence that
\begin{align*}
 \mu(\|\cdot\|_{\mathbb H^2}^p\wedge R)\le\liminf_{n\to\infty}\mu_{T_n}^0(\|\cdot\|_{\mathbb H^2}^p\wedge R)=\liminf_{n\to\infty}\frac{1}{T_n}\int_0^{T_n}\E[\|X(t;0)\|_{\mathbb H^2}^p\wedge R]\dd t\le C_p
\end{align*}
for some positive constant $C_p$ independent of $R$. By means of Fatou's lemma, we obtain
\begin{align*}
\mu(\|\cdot\|_{\mathbb H^2}^p)\le \liminf_{R\to \infty}\mu(\|\cdot\|_{\mathbb H^2}^p\wedge R)\le C_p.
\end{align*}
Hence, for any random variable $X_0$ whose distribution is $\mu$, one has $\E[\|X_0\|_{\mathbb H^2}^p]<\infty$ and $\mathbb P\circ({\bf \Pi}_N X_0)^{-1}=({\bf \Pi}_N)_{\#}\mu$. By \cref{prop:spatial-error,thm:time-error}, one can obtain
\begin{align*}
&\mathcal{W}_p\big(\mu P_{t_n}, ({\bf \Pi}_N)_{\#}\mu P_{t_n}^{N}\big)+\mathcal{W}_p\big(({\bf \Pi}_N)_{\#}\mu P_{t_n}^N, ({\bf \Pi}_N)_{\#}\mu P_n^{\tau,N}\big) \\
&\quad\le\,\|X(t_n;X_0)-X^N(t_n;{\bf \Pi}_N X_0)\|_{L^p(\Omega;\mathbb H^1)}+\|X^N(t_n;{\bf \Pi}_N X_0)-X^N_n({\bf \Pi}_N X_0)\|_{L^p(\Omega;\mathbb H^1)}\\
&\quad\le\,C_1(p)(1+t_n^{\frac12})\lambda_N^{-\frac12}+C_2(p)(1+t_n^{3})\tau^{\frac12},
\end{align*}
where $X^N_n({\bf \Pi}_N X_0)$ denotes the solution of \eqref{eq:full-scheme} with initial datum ${\bf \Pi}_N X_0$.

Hence, one arrives at
\begin{align*}
\mathcal{W}_p\big(({\bf \Pi}_N)_{\#}\mu, \mu_\tau^N \big)\le &\,\mathcal{W}_p\big(({\bf \Pi}_N)_{\#}\mu, \mu \big)+\mathcal{W}_p\big(\mu, \mu_\tau^N \big)\\
\le&\, \|X_0-{\bf \Pi}_NX_0\|_{L^p(\Omega;\mathbb H^1)}+\mathcal{W}_p\big(\mu, \mu_\tau^N \big)\\
\le&\, C\lambda_N^{-\frac12}\mu(\|\cdot\|_{\mathbb H^2}^p)^{\frac1p}+\mathcal{W}_p\big(\mu, \mu_\tau^N \big).
\end{align*}
Combining the above estimates with \eqref{eq:third-term} yields
\begin{align*}
\mathcal{W}_p\big(\mu, \mu_\tau^N \big)\le \tfrac{1+2\epsilon}{1-2\epsilon} \, e^{- \epsilon t_n} \mathcal{W}_p\big(\mu, \mu_\tau^N \big)+C_1(\epsilon,p)(1+t_n^{\frac12}+ e^{- \epsilon t_n})\lambda_N^{-\frac12}+C_2(p)(1+t_n^{3})\tau^{\frac12}.
\end{align*}
Choosing $\epsilon$ and $n$ such that $\frac{1+2\epsilon}{1-2\epsilon} e^{-\epsilon t_n}\le \frac12$, one obtains 
\begin{align*}  \mathcal{W}_p\big(\mu, \mu_\tau^N \big)\le\,C(p,\epsilon)\big(\lambda_N^{-\frac12}+\tau^{\frac12}\big).
\end{align*}
The proof is thus completed.
\end{proof}

The following result  is the weak error estimate of $\mu_\tau^N$ and $\mu$.
\begin{theorem}
Let conditions in \cref{prop:spatial-error} and \cref{ass:nonlinearity2} hold.
Then, for $\varphi \in C_{p,\gamma}$, one has for any $n\in\mathbb N$,
\begin{align*}
\Big| \int_{\mathbb H^1} \varphi \dd \mu - \int_{\mathbb H^1_N} \varphi \dd \mu_\tau^N \Big|\le C(\epsilon,p,\gamma)\|\varphi\|_{p,\gamma} \left(e^{- \epsilon \gamma t_n}+  \lambda_N^{-\frac \gamma2}(1+t_n^{\frac\gamma2})+  \tau^{\frac{\gamma+p/2}{2}}(1+t_n^{3(\gamma+\frac p2)})\right).
\end{align*}
\end{theorem}
\begin{proof}
We divide the error into two parts:
\begin{align*}
\Big| \int_{\mathbb H^1} \varphi \dd \mu - \int_{\mathbb H^1_N} \varphi \dd \mu_\tau^N \Big|
\le\Big| \int_{\mathbb H^1} \varphi \dd \mu - \int_{\mathbb H^1_N} \varphi \dd \mu^N \Big|+\Big| \int_{\mathbb H^1} \varphi \dd \mu^N - \int_{\mathbb H^1_N} \varphi \dd \mu_\tau^N \Big|.
\end{align*}
For any $t>0$, using the 
exponential ergodicity in  \eqref{eq:exponential-ergodicity} and \eqref{eq:exponential-ergodicity-SG}, the error estimate in \cref{prop:spatial-error} and the moment estimates in \cref{prop:H1-moment}, we arrive at
\begin{align*}
&\Big| \int_{\mathbb H^1} \varphi \dd \mu - \int_{\mathbb H^1_N} \varphi \dd \mu^N \Big|\\
&\quad\le\, \Big| \int_{\mathbb H^1} \varphi \dd \mu - P_t\varphi(0) \Big|+\Big| P_t\varphi(0) -P_t^N\varphi(0) \Big|+\Big| P_t^N\varphi(0) - \int_{\mathbb H^1_N} \varphi \dd \mu^N \Big|\\
&\quad\le \, \|\varphi\|_{p,\gamma}\E\left[\|X(t;0)-X^N(t;0)\|_{\mathbb H^1}^{\gamma}\big(1+\|X(t;0)\|_{\mathbb H^1}^p+\|X^N(t;0)\|_{\mathbb H^1}^p\big)^{\frac12}\right]\\
&\quad\quad+C(\epsilon,p,\gamma)\|\varphi\|_{p,\gamma}e^{- \epsilon\gamma t}\\
&\quad\le \, C(\epsilon,p,\gamma)\|\varphi\|_{p,\gamma}\left( e^{- \epsilon\gamma t}+ \lambda_N^{-\frac{\gamma}{2}}(1+t^{\frac\gamma2})\right).
\end{align*}
Similarly, for any $n\in\mathbb N$, using \eqref{eq:exponential-ergodicity-SG}, \eqref{eq:exp-full-dis}, \eqref{eq:mian-result}, \cref{prop:H1-moment} yields
\begin{align*}
&\,\Big| \int_{\mathbb H^1} \varphi \dd \mu^N - \int_{\mathbb H^1_N} \varphi \dd \mu_\tau^N \Big|\\
&\quad\le\, \Big| \int_{\mathbb H^1_N} \varphi \dd \mu^N - P_{t_n}^N\varphi(0) \Big|+\Big| P_{t_n}^N\varphi(0) -P_{n}^{\tau,N}\varphi(0) \Big|+\Big| P_{n}^{\tau,N}\varphi(0) - \int_{\mathbb H^1_N} \varphi \dd \mu^N_\tau \Big|\\
      &\quad\le \, \|\varphi\|_{p,\gamma}\E\left[\|X^N(t_n;0)-X^N_n\|_{\mathbb H^1}^{\gamma}\big(1+\|X^N(t_n;0)\|_{\mathbb H^1}^p+\|X^N_n\|_{\mathbb H^1}^p\big)^{\frac12}\right]\\
&\quad\quad+C(\epsilon,p,\gamma)\|\varphi\|_{p,\gamma}e^{- \epsilon\gamma t_n}\\
&\quad\le \, C(\epsilon,p,\gamma)\|\varphi\|_{p,\gamma}\left(\left(\E\left[\|X^N(t_n;0)-X^N_n\|_{\mathbb H^1}^{2\gamma+p}\right]\right)^{\frac12}+e^{- \epsilon\gamma t_n}\right)\\
       &\quad\le \, C(\epsilon,p,\gamma)\|\varphi\|_{p,\gamma} \left( e^{- \epsilon \gamma {t_n}}+ \tau^{\frac{\gamma+\frac p2}2} (1+{t_n}^{3(\gamma+\frac p2)})\right),
\end{align*}
which finishes the proof.
\end{proof}

\subsection{Strong law of large numbers}
This subsection presents the strong law of large numbers of the exact solution and the full discretization based on estimating the error between the corresponding time average and the ergodic limit $\mu(\varphi)$.
For any $X_0\in\mathbb H^2$, 
let $X(t)$ and $X_k^N$ respectively be the solutions of \eqref{SWE1} and \ref{eq:full-scheme} with initial data $X_0$ and ${\bf \Pi}_N X_0$ and $\mu$ be the unique invariant measures of \eqref{SWE1}. 

\begin{theorem}\label{thm:LLN-exact}
Let the conditions of \cref{prop:spatial-error} and \cref{ass:nonlinearity2} hold.
 Then for all $\varphi \in \mathcal C_{p,\gamma}$, the following strong law of large numbers holds
\begin{align}
\lim_{t\to\infty}\frac1t\int_{0}^t\varphi(X(s))\dd s=\mu(\varphi) \quad\text{a.s.}
\end{align}
\end{theorem}
\begin{proof}
With the help of the method in \cite[Proposition 2.6]{MR2222383}, we first claim that  for any $t \geq 1$,
\begin{equation}\label{eq:claim}
\left(\E\left[\left|\frac1t\int_0 ^t\varphi(X(s))\dd s-\mu(\varphi)\right|^{2 q}\right] \right)^{\frac{1}{2q}}\leq C(\epsilon,p,\gamma,q)\|\varphi\|_{p,\gamma}(1+\|X_0\|_{\mathbb H^1}^{\gamma+\frac p2})t^{-\frac12}.
\end{equation}
Indeed, without loss of generality, we assume that $\mu(\varphi)=0$.
Otherwise, we let 
$\tilde{\varphi} = \varphi - \mu(\varphi) $ and consider $\tilde{\varphi}$ instead of $\varphi$. 
Denoting
$$
I_q(t)=\sup _{0 \leq r \leq t} \E[\left|S_r(\varphi)\right|^{2 q} ],
\quad 
S_r(\varphi):=\int_0^r \varphi(X(s))\dd s,
$$
we have
\begin{align}
\E[\left|S_r(\varphi)\right|^{2 q}] & =\E \left[\int_{[0, r]^{2 q}} \prod_{i=1}^{2q}\varphi\left(X(r_i)\right)  \dd r_1 \cdots \dd r_{2 q}\right]\notag \\
& =(2 q) ! \E \left[\int_{\Delta_{[0,r]^{2q}}}  \prod_{i=1}^{2q}\varphi\left(X(r_i)\right)  \dd r_1 \cdots \dd r_{2 q}\right] \notag \\
& =(2 q) ! \E \left[\int_{\Delta_{[0,r]^{2q}}}  \prod_{i=1}^{2q-2}\varphi\left(X(r_i)\right) g\left(r_{2 q-1}, r_{2 q}\right) \dd r_1 \cdots \dd r_{2 q}\right] ,\label{eq:S^2p}
\end{align}
where we set
$$
\begin{aligned}
\Delta_{[0,r]^{2q}} &: =\left\{\left(r_1, \ldots, r_{2 q}\right) \in \R^{2 q}: 0 \leq r_1 \leq \cdots \leq r_{2 q} \leq r\right\}, \\
g\left(s_1, s_2\right) & :=\varphi\left(X(s_1)\right) \E\left[\varphi\left(X(s_2)\right) \mid \mathcal{F}_{s_1}\right], \quad s_1 \leq s_2 .
\end{aligned}
$$
The integral on the right-hand side of \eqref{eq:S^2p} can be represented as
$$
\begin{aligned}
& \int_{\Delta_{[0,r]^{2}}} g\left(r_{2 q-1}, r_{2 q}\right)\left\{\int_{\Delta_{[0,r_{2 q-1}]^{2(q-1)}}}  \prod_{i=1}^{2q-2}\varphi\left(X(r_i)\right) \dd r_1 \cdots \dd r_{2 q-2}\right\} \dd r_{2 q-1} \dd r_{2 q} \\
&\quad =\frac{1}{(2 q-2) !} \int_{\Delta_{[0,r]^{2}}} g\left(r_{2 q-1}, r_{2 q}\right)\left|S_{r_{2 q-1}}(\varphi)\right|^{2(q-1)} \dd r_{2 q-1} \dd r_{2 q}.
\end{aligned}
$$
Substituting this expression into \eqref{eq:S^2p} and applying H\"{o}lder's inequality, we obtain
$$
\E[\left|S_r(\varphi)\right|^{2 q}]\le C(q)\int_{\Delta_{[0,r]^{2}}}\left(\E\left[\left|S_{r_{2 q-1}}(\varphi)\right|^{2 q}\right]\right)^{\frac{q-1}{q}}\left(\E\left[\left|g\left(r_{2 q-1}, r_{2 q}\right)\right|^q\right]\right)^{\frac{1}{q}} \dd r_{2 q-1} \dd r_{2 q},
$$
where $C(q)=2 q(2 q-1)$. Taking the supremum over $r \in[0, t]$, we see that
$$
I_q(t) \leq C_q\left(I_q(t)\right)^{\frac{q-1}{q}} \int_{\Delta_{[0,t]^{2}}}\left(\E\left[\left|g\left(s_1, s_2\right)\right|^q\right]\right)^{\frac{1}{q}} \dd s_1 \dd s_2 .
$$
Thus, we have 
\begin{align*}
I_q(t)\leq\left(C_q \int_{\Delta_{[0,t]^{2}}}\left(\E[\left|g\left(s_1, s_2\right)\right|^q]\right)^{\frac{1}{q}} \dd s_1 \dd s_2\right)^q
\end{align*}
and
\begin{align}
\label{eq:I_q} \E\left[\left|\frac1t\int_0^t\varphi(X(s)) \dd s\right|^{2q}\right]\leq\left(\frac{C_q}{t^2} \int_{\Delta_{[0,t]^{2}}}\left(\E[\left|g\left(s_1, s_2\right)\right|^q]\right)^{\frac{1}{q}} \dd s_1 \dd s_2\right)^q.
\end{align}
It follows from the Markov property and inequality \eqref{eq:exponential-ergodicity} that
\begin{align*}
|\E\left[\varphi(X(s_2))\mid \mathcal{F}_{s_1}\right]|&=| P_{s_2-s_1} \varphi\left(X(s_1)\right)|
\\ & \leq C(\epsilon,p)\|\varphi\|_{p,\gamma}e^{-\epsilon\gamma (s_2-s_1)}(1+\|X(s_1)\|_{\mathbb H^1}^{\gamma+\frac p2}).
\end{align*}
Using the definition of $\mathcal C_{p,\gamma}$, we obtain
\begin{align*}
\left|g\left(s_1, s_2\right)\right| & \leq C(\epsilon,p,\gamma)\|\varphi\|_{p,\gamma}e^{-\epsilon\gamma (s_2-s_1)}(1+\|X(s_1)\|_{\mathbb H^1}^{\gamma+\frac p2})\left|\varphi\left(X(s_1)\right)\right|  \\
& \leq C(\epsilon,p,\gamma)\|\varphi\|_{p,\gamma}^2e^{-\epsilon\gamma (s_2-s_1)}(1+\|X(s_1)\|_{\mathbb H^1}^{2\gamma+p}),
\end{align*}
which implies that
$$
\begin{aligned}
\left(\E[\left|g\left(s_1, s_2\right)\right|^q]\right)^{\frac{1}{q}} & \leq C(\epsilon,p,\gamma)\|\varphi\|_{p,\gamma}^2e^{-\epsilon\gamma (s_2-s_1)}\big(1+(\E[\|X(s_1)\|_{\mathbb H^1}^{(2\gamma+p)q}])^\frac1q\big) \\
& \leq C(\epsilon,p,\gamma)\|\varphi\|_{p,\gamma}^2e^{-\epsilon\gamma (s_2-s_1)} (1+\|X_0\|_{\mathbb H^1}^{2\gamma+p}).
\end{aligned}
$$
Next, we calculate that
\begin{align*}
&\int_{\Delta_{[0,t]^{2}}}\left(\E[\left|g\left(s_1, s_2\right)\right|^q]\right)^{\frac{1}{q}} \dd s_1 \dd s_2\\
&\quad\le\, C(\epsilon,p,\gamma)\|\varphi\|_{p,\gamma}^2 (1+\|X_0\|_{\mathbb H^1}^{2\gamma+p})\int_{\Delta_{[0,t]^{2}}}e^{-\epsilon\gamma (s_2-s_1)}\dd s_1\dd s_2\\
&\quad\le\,C(\epsilon,p,\gamma)\|\varphi\|_{p,\gamma}^2 (1+\|X_0\|_{\mathbb H^1}^{2\gamma+p})t.
\end{align*}
Substituting this inequality into \eqref{eq:I_q} leads to \eqref{eq:claim}.
By virtue of the Borel--Cantelli Lemma, one has that for $\varepsilon>0$,
\begin{align}\label{eq:pathwise-error}
\left|\frac{1}{t}\int_0 ^{t}\varphi(X(t))\dd t-\mu(\varphi)\right| \leq K_1(\varepsilon,p,\gamma,\|\varphi\|_{p,\gamma})t^{-\frac12+\varepsilon}\quad \text{a.s.}
\end{align}
for some random variable $K_1(\varepsilon,p,\gamma,\|\varphi\|_{p,\gamma})\in \cap_{q\ge1}L^q(\Omega)$. Letting $t\to\infty$ finishes the proof.
\end{proof}

\begin{theorem}\label{cor:LLN}
Let the conditions of \cref{prop:spatial-error} and \cref{ass:nonlinearity2} hold.
 Then for all $\varphi \in \mathcal C_{p,\gamma}$, the following strong law of large numbers holds
\begin{align}
\lim_{n\to\infty}\lim_{\tau\to 0}\lim_{N\to \infty}\frac1n\sum_{k=1}^n\varphi(X^N_k)=\mu(\varphi) \quad\text{a.s.}
\end{align}
\end{theorem}

The proof of this theorem comes directly from the following proposition, which is  the error estimate between the time average of the numerical  solution of \eqref{eq:full-scheme} and the ergodic limit.

\begin{proposition}\label{thm:approximation-limit}
Under conditions in \cref{cor:LLN},  for all $\varphi \in \mathcal C_{p,\gamma}$ and all $\varepsilon>0$ there exists a positive random variable $K(p,\gamma,\varepsilon,\|\varphi\|_{p,\gamma})$ such that
\begin{align*}
\Big|\frac1n & \sum_{k=1}^n\varphi(X^N_k)-\mu(\varphi)\Big|
\\ & \le K(p,\gamma,\varepsilon,\|\varphi\|_{p,\gamma}) \left(\tau^{\frac{\gamma+p/2}{2}-\varepsilon}\big(1+t_n^{3(\gamma+\frac p2)}\big)^{1+\varepsilon}+\lambda_N^{-\frac\gamma2+\varepsilon}\big(1+t_n^{\frac\gamma2}\big)^{1+\varepsilon}+t_n^{-\frac12+\varepsilon}\right)\quad \text { a.s. }
\end{align*}
for all $n \in \mathbb{N}$. Moreover, $\mathbb{E}\left[\left|K(p,\gamma,\varepsilon,\|\varphi\|_{p,\gamma})\right|^q\right]<\infty$ for all $q \in \mathbb{N}$.
\end{proposition}
\begin{proof}
For any fixed functional $\varphi \in \mathcal C_{p,\gamma}$, one has
\begin{align*}
\Big|\frac1n\sum_{k=1}^n\varphi(X^N_k)- &\mu(\varphi)\Big|\le \left|\frac1n\sum_{k=1}^n\varphi(X^N_k)-\frac{1}{t_n}\int_{0}^{t_n}\varphi(X^N(t))\dd t\right|\\
&\,+ \left|\frac{1}{t_n}\int_{0}^{t_n}\big[\varphi(X^N(t))-\varphi(X(t))\big]\dd t\right|+\left|\frac{1}{t_n}\int_{0}^{t_n}\varphi(X(t))\dd t-\mu(\varphi)\right|.
\end{align*}
Here $X(t)$ denotes the solution of \eqref{SWE1} with initial datum $X_0$ and $X^N(t)$ denotes the solution of \eqref{eq:spectral-galerkin} with initial datum ${\bf \Pi}_N X_0$.

Using H\"{o}lder's inequality and \cref{prop:spatial-error}, one obtains
\begin{align*}
&\E\left[\left|\frac{1}{t_n}\int_{0}^{t_n}\big[\varphi(X^N(t))-\varphi(X(t))\big]\dd t\right|^q\right]\\
&\quad\le\,\frac{1}{t_n}\int_{0}^{t_n}\E\left[\left|\varphi(X^N(t))-\varphi(X(t))\right|^q\right]\dd t\\
&\quad\le\,\frac{\|\varphi\|_{p,\gamma}}{t_n}\int_{0}^{t_n}\E\left[\|X^N(t)-X(t)\|_{\mathbb H^1}^{\gamma q}(1+\|X^N(t)\|_{\mathbb H^1}^p+\|X(t)\|_{\mathbb H^1}^p)^{\frac q2}\right]\dd t\\
&\quad\le\,C(q,p,\gamma,\|X_0\|_{\mathbb H^2})\|\varphi\|_{p,\gamma}\lambda_{N}^{-\frac{\gamma q}{2}} \big(1+t_n^{\frac\gamma2}\big)^q.
\end{align*}
This together with the Markov inequality implies that for any $\varepsilon > 0$ and all $\delta > 0$ 
\begin{align*}
&\mathbb P\left(\left|\frac{1}{t_n}\int_{0}^{t_n}\big[\varphi(X^N(t))-\varphi(X(t))\big]\dd t\right|>\delta\lambda_N^{-\frac{\gamma}{2}+\varepsilon}(1+t_n^{\frac\gamma2})^{1+\varepsilon}\right) \\
&\quad\le \E\left[\left|\frac{1}{t_n}\int_{0}^{t_n}\big[\varphi(X^N(t))-\varphi(X(t))\big]\dd t\right|^q\right]\lambda_N^{\frac{\gamma q}{2}-q\varepsilon}(1+t_n^{\frac\gamma2})^{-q(1+\varepsilon)}\delta^{-q}\\
&\quad\le C(q,p,\gamma,\|X_0\|_{\mathbb H^2})\|\varphi\|_{p,\gamma}(1+t_n^{\frac\gamma2})^{-q\varepsilon}\lambda_N^{-q\varepsilon}\delta^{-q}.
\end{align*}
Then for $q>\frac1\varepsilon\max\big\{\frac d2,\frac{\gamma}{2}\big\}$,
\begin{align*}
&\sum_{n=1}^\infty\sum_{N=1}^{\infty}\mathbb P\left(\left|\frac{1}{t_n}\int_{0}^{t_n}\big[\varphi(X^N(t))-\varphi(X(t))\big]\dd t\right|>\delta\lambda_N^{-\frac{\gamma}{2}+\varepsilon}(1+t_n^{\frac\gamma2})^{1+\varepsilon}\right)\\\
&\quad\leq \,C(q,p,\gamma,\|X_0\|_{\mathbb H^2})\|\varphi\|_{p,\gamma}\delta^{-q}\sum_{n=1}^\infty(1+t_n^{\frac\gamma2})^{-q\varepsilon} \sum_{N=1}^{\infty} \lambda_N^{-q\varepsilon}<\infty.
\end{align*}
Hence, by the Borel--Cantelli lemma, the random variable 
$$
K_2(p,\gamma,\varepsilon,\|\varphi\|_{p,\gamma}):=\sup _{n,N>0} \left|\frac{1}{t_n}\int_{0}^{t_n}\big[\varphi(X^N(t))-\varphi(X(t))\big]\dd t\right|\lambda_N^{\frac{\gamma}{2}-\varepsilon}(1+t_n^{\frac\gamma2})^{-1-\varepsilon}
$$ 
is a.s. finite, which   implies 
\begin{align*}
 \left|\frac{1}{t_n}\int_{0}^{t_n}\big[\varphi(X^N(t))-\varphi(X(t))\big]\dd t\right| \le K_2(p,\gamma,\varepsilon,\|\varphi\|_{p,\gamma})\lambda_N^{-\frac\gamma2+\varepsilon}(1+t_n^{\frac\gamma2})^{1+\varepsilon}.
\end{align*}
For $q>\frac1\varepsilon\max\big\{\frac d2,\frac\gamma 2\big\}$, it holds that
\begin{align*}
&\E\left[\left|K_2(p,\gamma,\varepsilon,\|\varphi\|_{p,\gamma})\right|^q\right] \\
&\quad\le \,\sum_{n=1}^{\infty} \sum_{N=1}^{\infty} \E\left[\left|\frac{1}{t_n}\int_{0}^{t_n}\big[\varphi(X^N(t))-\varphi(X(t))\big]\dd t\right|^q\right]\lambda_N^{\frac{\gamma q}{2}-q\varepsilon}(1+t_n^{\frac\gamma2})^{-q(1+\varepsilon)}\\
&\quad<\,\infty.
\end{align*}
Using Jensen's inequality gives $\E\left[\left|K_2(p,\gamma,\varepsilon,\|\varphi\|_{p,\gamma})\right|^q\right]<\infty$ for all $q \geq 1$. 

Similarly, by using \cref{thm:time-error} and the Borel--Cantelli Lemma,  one  shows that there exists a random variable $K_3(p,\gamma,\varepsilon,\|\varphi\|_{p,\gamma})\in \cap_{q\ge1}L^q(\Omega)$ such that
\begin{align*}
\left|\frac1n\sum_{k=1}^n\varphi(X^N_k)-\frac{1}{t_n}\int_{0}^{t_n}\varphi(X^N(t))\dd t\right| \le K_3(p,\gamma,\varepsilon,\|\varphi\|_{p,\gamma})\tau^{\frac{\gamma+p/2}{2}-\varepsilon}\big(1+t_n^{3(\gamma+\frac p2)}\big)^{1+\varepsilon}.
\end{align*}
Combined the above estimates with \eqref{eq:pathwise-error} yields the desired estimate with
$
K(\varepsilon,p,\gamma,\|\varphi\|_{p,\gamma}):=\max_{i=1}^3 K_i(\varepsilon,p,\gamma,\|\varphi\|_{p,\gamma}).
$
\end{proof}

\bibliographystyle{plain}
\bibliography{references}

\end{document}